\documentclass[leqno, 11pt, a4paper]{amsart}

\usepackage{amssymb}
\usepackage{a4wide}
\usepackage{centernot}
\usepackage{bbm}
\usepackage[cal=euler, scr=boondoxo]{mathalfa}
\usepackage[colorlinks=true, linkcolor=blue,citecolor=blue]{hyperref}
\usepackage{graphicx}
\newtheorem{theorem}{Theorem}[section]
\newtheorem{lemma}[theorem]{Lemma}
\newtheorem{prop}[theorem]{Proposition}

\newtheorem{corollary}[theorem]{Corollary}
\theoremstyle{definition}

\theoremstyle{remark}
\newtheorem{remark}[theorem]{Remark}

\numberwithin{equation}{section}

\newcommand{\lip}{\mathrm{Lip}}

\newcommand{\capa}{\mathrm{cap}}
\DeclareMathOperator{\supp}{\mathrm{supp}}

\newcommand{\De}{\mathrm{d}}

\newcommand{\cA}{\ensuremath{\mathcal A}}
\newcommand{\cB}{\ensuremath{\mathcal B}}
\newcommand{\cC}{\ensuremath{\mathcal C}}
\newcommand{\cD}{\ensuremath{\mathcal D}}
\newcommand{\cE}{\ensuremath{\mathcal E}}

\newcommand{\cI}{\ensuremath{\mathcal I}}

\newcommand{\cK}{\ensuremath{\mathcal K}}
\newcommand{\cL}{\ensuremath{\mathcal L}}
\newcommand{\cM}{\ensuremath{\mathcal M}}

\newcommand{\cR}{\ensuremath{\mathcal R}}
\newcommand{\cS}{\ensuremath{\mathcal S}}

\newcommand{\cU}{\ensuremath{\mathcal U}}
\newcommand{\cV}{\ensuremath{\mathcal V}}
\newcommand{\cW}{\ensuremath{\mathcal W}}

\newcommand{\bbE}{\ensuremath{\mathbb E}}

\newcommand{\bbL}{\ensuremath{\mathbb L}}

\newcommand{\bbP}{\ensuremath{\mathbb P}}

\newcommand{\bbR}{\ensuremath{\mathbb R}}

\newcommand{\bbZ}{\ensuremath{\mathbb Z}}
\newcommand{\scrL}{\mathscr{L}}
\newcommand{\scrM}{\mathscr{M}}
\newcommand{\scrh}{\mathscr{h}}

\begin{document}

\title[Entropic repulsion by disconnection]{Entropic repulsion for the occupation-time field of random interlacements conditioned on disconnection}


\author{Alberto Chiarini}
\address{Mathematics Department, UCLA}
\curraddr{520, Portola Plaza, 90095 Los Angeles, USA}
\email{chiarini@math.ucla.edu}
\thanks{}

\author{Maximilian Nitzschner}
\address{Departement Mathematik, ETH Z\"urich}
\curraddr{101, R\"amistrasse, CH-8092 Z\"urich, Switzerland}
\email{maximilian.nitzschner@math.ethz.ch}
\thanks{}

\begin{abstract}
 We investigate percolation of the vacant set of random interlacements on $\bbZ^d$, $d\geq 3$, in the strongly percolative regime. 
 We consider the event that the interlacement set at level $u$ disconnects the discrete blow-up of a compact set $A\subseteq \bbR^d$ from the boundary of an enclosing box. We derive asymptotic large deviation upper bounds on the probability that the local averages of the occupation times deviate from a specific function depending on the harmonic potential of $A$, when disconnection occurs. 
 If certain critical levels coincide, which is plausible but open at the moment, these bounds imply that conditionally on disconnection, the occupation-time profile undergoes an entropic push governed by a specific function depending on $A$. Similar entropic repulsion phenomena conditioned on disconnection by level-sets of the discrete Gaussian free field on $\bbZ^d$, $d \geq 3$, have been obtained by the authors in~\cite{chiarini2018entropic}. Our proofs rely crucially on the `solidification estimates' developed in~\cite{nitzschner2017solidification} by A.-S.\ Sznitman and the second author.
\end{abstract}

\subjclass[2010]{}
\date{}
\dedicatory{}
\maketitle

\section{Introduction}
\label{sec:introduction}
Random interlacements have been introduced to understand the kind of disconnection or fragmentation created by a simple random walk, and constitute a percolation model with long-range dependence and non-trivial percolative properties, see e.g.~\cite{ sznitman2009domination, sznitman2010vacant, teixeira2011fragmentation}.
This article aims at understanding the optimal way for random interlacements on $\bbZ^d$, $d \geq 3$, to disconnect the discrete blow-up of a compact set from an enclosing box, when their vacant set is in a strongly percolative regime. 

Specifically, we consider for the discrete blow-up of a  compact set $A \subseteq \mathbb{R}^d$ the \textit{disconnection event} that random interlacements isolate it from the boundary of an enclosing box. Our goal is to track the behavior that conditioning on disconnection entails for the occupation-time profile of random interlacements. As a main result, we derive an asymptotic large deviation upper bound on the probability of the event that the average of the occupation-time profile deviates from a certain function involving the harmonic potential of $A$, when disconnection occurs. Large deviation results on the probability of the disconnection event itself have been obtained in \cite{li2014lowerBound}, concerning lower bounds, and in \cite{sznitman2015disconnection} for upper bounds in the case where $A$ is itself a box. The latter were later generalized to arbitrary compact sets $A$ in \cite{nitzschner2017solidification} by making use of `solidification estimates', a technique that is also pivotal in this work. It is plausible but open at the moment that certain critical levels for the percolation of the vacant set of random interlacements coincide. If this is the case and the set $A$ is regular, the upper and lower bounds of the references given above would match in principal order and yield the exact asymptotic behavior for the probability of the disconnection event. Under the same circumstances, the results put forward in this work imply that conditioning on disconnection will effectively force the occupation times of the random interlacements to be pinned locally to $(\sqrt{u} + (\sqrt{u_\ast} - \sqrt{u})\scrh_{{A}}(\cdot/N))^2$, where $\scrh_A$ is the harmonic potential of the set $A$, $u$ is the level of the random interlacements under consideration and $u_\ast$ is the critical level for percolation of the vacant set. This shift in the local level of the occupation-time profile conditionally on disconnection should be compared with the `strategy' used in~\cite{li2014lowerBound} to enforce disconnection, making use of so-called \textit{tilted interlacements} and provides further evidence that this object emerges naturally when conditioning random interlacements on disconnection.

The upward shift for the occupation-time profile can be understood in the context of \textit{entropic repulsion} phenomena, and we are guided by similar findings for the Gaussian free field conditioned on disconnection by level-sets in a strongly percolative regime, see \cite{chiarini2018entropic}, which extend a more elementary result from \cite{nitzschner2018entropic}. Namely, if certain critical levels coincide, forcing the excursion set of the Gaussian free field below a given level to disconnect the discrete blow-up of a regular set $A$ from the boundary of a box, lowers the field by an amount proportional to the harmonic potential of $A$. This behavior is reminiscent of the study of classical entropic repulsion phenomena, which focuses on a Gaussian free field conditioned to be positive over a given set, see for instance \cite{bolthausen2001entropic, bolthausen1995entropic, deuschel1999pathwise}. 

\vspace{\baselineskip}

We will now describe the model and our results in a more detailed way. Consider $\bbZ^d$, $d \geq 3$. For a given $u \geq 0$ we let $\cI^u$ stand for continuous-time random interlacements at level $u$ in $\bbZ^d$, which are governed by some probability measure $\bbP$. The vacant set at level $u$ is denoted by $\cV^u = \bbZ^d \setminus \cI^u$. For a thorough introduction and background on the model we refer to~\cite{drewitz2014introduction}. There are three critical levels $0 < \overline{u} \leq u_\ast \leq u_{\ast\ast} < \infty$ in the study of the percolation of the vacant set. The strongly non-percolative regime for $\cV^u$ corresponds to $u > u_{\ast\ast}$, and the strongly percolative regime to $0 < u < \overline{u}$ (the positivity of $\overline{u}$ was proved in~\cite{drewitz2014local}). We refer to (1.2) and (3.3) of \cite{sznitman2015disconnection} for the precise definition of these levels. The level $u_\ast$ corresponds to the threshold of percolation for the vacant set of random interlacements. Although plausible, it is still an open question to show that these three levels are in fact equal (progress towards showing $u_\ast = u_{\ast\ast}$ might come from~\cite{duminil2017sharp}).
\vspace{\baselineskip} 

We consider a compact set $A \subseteq \bbR^d$ with non-empty interior which is contained in the interior of a closed box of side-length $2M$, $M > 0$, centered at the origin. For an integer $N \geq 1$, we define 
\begin{equation}
A_N = (NA) \cap \bbZ^d \text{ and } S_N = \lbrace x \in \bbZ^d ; |x|_\infty = \lfloor MN\rfloor \rbrace
\end{equation}
(where $\lfloor\cdot\rfloor$ denotes the integer part and $| \cdot|_\infty$ the sup-norm of a vector in $\bbR^d$), which are the discrete blow-up of $A$ and the boundary of the discrete blow-up of the box that contains $A$. In what follows, we study the \textit{disconnection~event} 
\begin{equation}
 \label{eq:DisconnectionEvent}
 \cD^u_N = \Big\lbrace A_N \stackrel{\cV^u}{\centernot\longleftrightarrow} S_N \Big\rbrace,
 \end{equation}
 which stands for the absence of a path in $\cV^u$ that connects $A_N$ to $S_N$. The asymptotic leading order behavior of $\bbP[\cD^u_N]$ has been obtained in~\cite{li2014lowerBound,nitzschner2017solidification}. On one hand, Theorem 0.1 of \cite{li2014lowerBound} gives the lower bound for $u < u_{\ast\ast}$
 \begin{equation}
 \label{eq:LowerBoundInterlacements}
 \liminf_N \frac{1}{N^{d-2}} \log \bbP[\cD^u_N] \geq - \frac{1}{d}(\sqrt{u_{\ast\ast}} - \sqrt{u})^2 \capa(A),
 \end{equation}
 where $\capa(\cdot)$ stands for the Brownian capacity (see for instance~\cite{port2012brownian}, p.58 for a definition). 
 
 On the other hand, Theorem 4.1 of \cite{nitzschner2017solidification} also provides us with an upper bound for $u< \overline{u}$, namely
 \begin{equation}
 \label{eq:UpperBoundInterlacements}
 \limsup_N \frac{1}{N^{d-2}} \log \bbP[\cD^u_N] \leq - \frac{1}{d}(\sqrt{\overline{u}} - \sqrt{u})^2 \capa(\mathring{A}), 
\end{equation}  
where $\mathring{A}$ denotes the interior of the set $A$. As mentioned, if $A$ is \textit{regular} in the sense that $\capa(A) = \capa(\mathring{A})$ and if the critical levels $\overline{u}$, $u_\ast$ and $u_{\ast\ast}$ are shown to be equal, then the right hand sides of \eqref{eq:LowerBoundInterlacements} and \eqref{eq:UpperBoundInterlacements} coincide.  
\vspace{\baselineskip}

The proof of the lower bound \eqref{eq:LowerBoundInterlacements} is based on a change of probability method and involves the use of probability measures $\widetilde{\bbP}_N$ governing \textit{tilted interlacements}. The choice of $\widetilde{\bbP}_N$ corresponds in essence to a certain strategy to enforce disconnection --- roughly speaking, under $\widetilde{\bbP}_N$ the interlacements follow a slowly space-modulated intensity equal to $\scrM^u_A(\tfrac{x}{N}) = (\sqrt{u} + (\sqrt{u_{\ast\ast}} - \sqrt{u}) \scrh_A(\tfrac{x}{N}))^2$, $x \in \bbZ^d$, which informally creates a ``fence'' around $A_N$, where they locally behave as interlacements at a level $u_{\ast\ast}$ (one actually chooses a level slightly above $u_{\ast\ast}$ in the construction). Thus, the tilted interlacements are in a strongly non-percolative regime in the vicinity of $A_N$, and disconnect this set from $S_N$ with high probability, or in other words, $\cD^u_N$ becomes typical under $\widetilde{\bbP}_N$. The fact that the lower and upper bounds coincide for regular sets, if the critical levels are the same, hints at a certain optimality of the tilted interlacements: Whenever the rare event $\cD^u_N$ occurs, we expect the random interlacements to effectively behave like tilted interlacements. Our main results~\eqref{eq:intro_main_result} and~\eqref{eq:intro_corollary} provide further evidence that this reasoning is indeed correct. 
\begin{figure}[ht]\label{fig:pushup}
    \centering
    \includegraphics[width=0.8\textwidth]{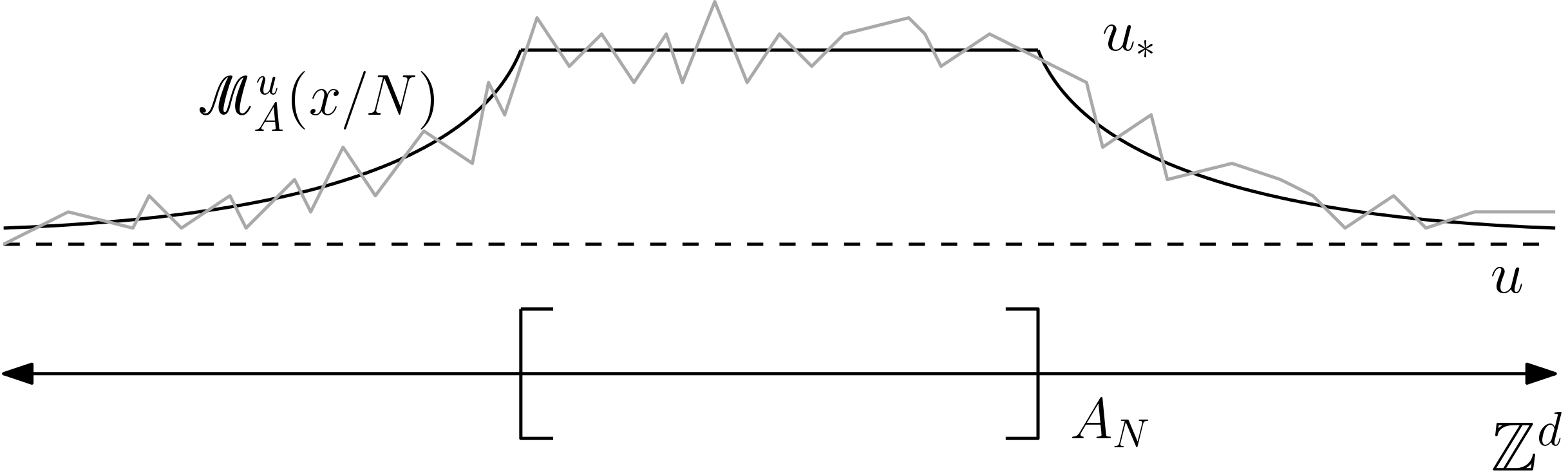}     \caption{Occupation-time field conditioned on disconnection.}
\end{figure}

We introduce the random measure on $\bbR^d$
\begin{equation}
\scrL_{N,u} = \frac{1}{N^d} \sum_{x \in \bbZ^d} L_{x,u} \delta_{\tfrac{x}{N}}, 
\end{equation}
where $(L_{x,u})_{x \in \bbZ^d}$ stands for the field of occupation times of continuous-time random interlacements at level $u > 0$ (see Section~\ref{sec:notation} for details) and we define for any continuous, compactly supported function $\eta: \bbR^d \rightarrow \bbR$, and any signed Radon measure $\nu$ on $\bbR^d$
\begin{equation}
\langle \nu, \eta \rangle = \int \eta(x) \nu(\De x).
\end{equation}
Moreover, if $\nu(\De x) = f(x) \De x$ holds, we write $\langle f, \eta \rangle$ instead of $\langle \nu, \eta \rangle$.
For two non-negative Radon measures $\mu$, $\nu$ on $B_R\subseteq \bbR^d$, a closed box of side length $2R >0$ centered at the origin, we denote by $d_R(\mu,\nu)$ the sum of $|\mu(B_R) - \nu(B_R)|$ and the $1$-Wasserstein distance among the probability measures obtained by normalizing $\nu$ and $\mu$ by their respective total masses, see \eqref{Distance} for a precise definition.

Finally, we introduce for $B \subseteq \bbR^d$ open or closed the function
\begin{equation}
\label{eq:DefinitionScrL}
\scrM^u_B(x) = \big(\sqrt{u} + (\sqrt{\overline{u}} - \sqrt{u})\scrh_{B}(x)\big)^2,\qquad x\in \bbR^d,
\end{equation}
with $\scrh_B$ the harmonic potential of $B$ (see~\eqref{eq:HarmonicPot}).

Our main result comes in Theorem \ref{thm:MainResult} and states that for $u < \overline{u}$, $\Delta > 0$ and any $R>0$ with $[-M,M]^d \subseteq B_R$, one has
\begin{equation}\label{eq:intro_main_result}
    \begin{aligned}
\limsup_{N } \frac{1}{N^{d-2}} \log \bbP &\Big[d_R(\scrL_{N,u},\scrM^u_{\mathring{A}}) \geq \Delta ; \cD^u_N \Big] \\ &\leq - \frac{1}{d}(\sqrt{\overline{u}} - \sqrt{u})^2 \capa(\mathring{A}) - c_1(\Delta, R, A,u),
    \end{aligned}
\end{equation}
where $c_1(\Delta, R, A,u)$ is a positive constant depending on $\Delta, R, A, u$ and also on $d$, which fulfills $c_1(\Delta, R, A,u) \sim c_2(\Delta, R, A)\sqrt{u}$ as $u \rightarrow 0$, where $c_2(\Delta, R, A) > 0$. Moreover, we show in Corollary \ref{thm:CorollaryPush} that if $\overline{u} = u_\ast = u_{\ast\ast}$ holds and $\capa(A) = \capa(\mathring{A})$, one has an asymptotic result for the conditional measure given disconnection, which states that for $u < u_\ast$ and $R>0$ fulfilling $[-M,M]^d \subseteq B_R$, it holds that
\begin{equation}
\label{eq:intro_corollary}
\lim_{N} \bbE\Big[d_R(\scrL_{N,u},\scrM^u_{\mathring{A}})  \wedge 1 \big\vert  \cD^u_N \Big] = 0.
\end{equation}

If we fix $R>0$ large enough so that $[-M,M]^d \subseteq B_R$, then, $\bbP[\ \cdot\ | \cD^u_N]$-almost surely, one has $\scrL_{N,u}(B_R) > 0$, thus, in view of the definition of $d_R$, we can rephrase the above statement as follows: conditionally on $\cD^u_N$, the random measure $\scrL_{N,u}$ converges weakly in probability to $\scrM^u_A$ when restricted to $B_R$. In other words, local averages of the occupation-time field are pinned to~$\scrM^u_A$. 
\vspace{\baselineskip}

As mentioned earlier, random interlacements were introduced to study the disconnection or fragmentation created by a random walk, which itself can be seen heuristically as the limit as $u \rightarrow 0$ of random interlacements. Therefore, it is natural to ask whether the occupation-time field of the random walk experiences a similar entropic push when it is conditioned to isolate the blow-up of the macroscopic body $A_N$ from the boundary of the enclosing box $S_N$. In the case of large deviation upper bounds for the disconnection by random walk, a coupling argument was pertinent to infer the leading order behavior for the disconnection probability as the limit when $u \rightarrow 0$ of the equivalent quantity for random interlacements, see~Corollary 6.4 of \cite{sznitman2015disconnection} and Corollary 4.4 of~\cite{nitzschner2017solidification}. One may hope that the results of this article provide some insight into the `random walk conditioned on disconnection', see also Remark \ref{Remark_RW}.

\vspace{\baselineskip}
The article is organized as follows: In Section~\ref{sec:notation}, we introduce further notation and recall useful results about random walks and random interlacements, some potential theory and solidification estimates from \cite{nitzschner2017solidification}. In Section~\ref{sec:occupationtimebound}, we prove an exponential upper bound for the occupation time of a perturbed potential, which is instrumental in the proof of the main result. In Section~\ref{sec:entropicpush}, we state and prove Theorem~\ref{thm:MainResult}, which corresponds to the entropic repulsion under disconnection for the occupation-time field~\eqref{eq:intro_main_result}. In the Appendix, we provide in Proposition \ref{ComparisonCapacities} an asymptotic comparison between Brownian capacities of certain well-separated finite collection of boxes in $\bbR^d$ of similar sizes.
\vspace{\baselineskip}

We conclude this introduction with our convention regarding constants. We denote by $c,\,c',\,\ldots$ positive constants changing from place to place. Numbered constants $c_1,\,c_2,\ldots$ will refer to the value assigned to them when they first appear in the text and dependence on additional parameters is indicated in the notation. All constants may depend implicitly on the dimension.

\section{Notation and useful results}\label{sec:notation}

In this section we introduce some notation and collect useful results concerning random walks, potential theory, random interlacements and the solidification estimates for porous interfaces from~\cite{chiarini2018entropic} and~\cite{nitzschner2017solidification}. These solidification estimates will be pivotal in the following sections to derive the large deviation upper bound~\eqref{eq:intro_main_result}. We will assume that $d \geq 3$ throughout the article.
\vspace{\baselineskip}

We start by introducing some notation. For real numbers $s,t$, we denote by $s \vee t$ and $s \wedge t$ the maximum and minimum of $s$ and $t$, respectively, and we denote the integer part of $s$ by $\lfloor s \rfloor$. We consider on $\mathbb{R}^d$ the Euclidean and $\ell^\infty$-norms $|\cdot|$ and $|\cdot|_\infty$ and  the corresponding closed balls $B_2(x,r)$ and $B_\infty(x,r)$ of radius $r\geq 0$ and center $x \in \bbR^d$. Also,  we denote by $B(x,r) = \lbrace y \in \mathbb{Z}^d; |x-y|_\infty \leq r \rbrace \subseteq \bbZ^d$ the closed $\ell^\infty$-ball of radius $r\geq 0$ and center $x\in \bbZ^d$.  
For subsets $G,H \subseteq \mathbb{R}^d$, we denote by $d(G,H)$ their mutual $\ell^\infty$-distance, i.e.\ $d(G,H) = \inf \lbrace |x-y|_\infty ; x \in G , y \in H \rbrace$ and write for simplicity $d(x,G)$ instead of $d(\lbrace x \rbrace, G)$ for $x \in \mathbb{R}^d$. For $K \subseteq \mathbb{Z}^d$, we let $|K|$ denote the cardinality of $K$.
If $x,y \in \mathbb{Z}^d$ fulfill $|x - y| = 1$, we call them neighbors and write $x \sim y$. We call $\pi : \lbrace 0,\ldots, N \rbrace \rightarrow \mathbb{Z}^d$ a nearest neighbor path (of length $N \geq 1$) if $\pi(i) \sim \pi(i+1)$ for all $0 \leq i \leq N-1$. For subsets $K,K',U \subseteq \mathbb{Z}^d$, we write $K \stackrel{U}{\leftrightarrow} K'$ (resp.\ $K \stackrel{U}{\nleftrightarrow} K'$) if there is a path $\pi$ with values in $U$ starting in $K$ and ending in $K'$ (resp.\ if there is no such path) and we say that $K$ and $K'$ are connected in $U$ (resp.\ not connected in $U$). Given two measurable, real-valued functions $f,g$ on $\bbR^d$ such that $|fg|$ is Lebesgue-integrable we define $\langle f,g\rangle = \int f(y)g(y)dy$. For functions $f:\bbR^d \rightarrow \bbR$ and $h: \bbZ^d \rightarrow \bbR$, we denote by $\|f\|_\infty$ and $\|h \|_\infty$ the respective supremum norms over $\bbR^d$ and $\bbZ^d$, and we denote by $f^+ = f \vee 0$ and $f^- = (-f) \vee 0$ the positive and negative part of $f$, respectively. If $f:\bbR^d\to \bbR$ is continuous and compactly supported and $\nu$ is a Radon measure on $\bbR^d$ we write $\langle \nu, f \rangle = \int f\De \nu$. Similarly, for functions $u, v : \bbZ^d \rightarrow \bbR$, we will routinely write $\langle u, v\rangle_{\bbZ^d} = \sum_{ z \in \bbZ^d} u(z)v(z)$, if $|uv|$ is summable.
\vspace{\baselineskip}

We now introduce some path spaces and the set-up for the continuous-time simple random walk on $\bbZ^d$.  We denote by $\widehat{W}_+$ and $\widehat{W}$ the spaces of infinite and doubly-infinite $\bbZ^d \times (0,\infty)$-valued sequences, such that the first coordinate sequence forms a nearest-neighbor path in $\bbZ^d$, spending finite time in any finite subset of $\bbZ^d$, and the sequence of second coordinates (interpreted as time spent at a lattice site) has an infinite sum, respectively infinite forward and backwards sums. We denote by $\widehat{\cW}_+$ and $\widehat{\cW}$ the respective $\sigma$-algebras generated by the coordinate maps. The measure $P_x$ is a law on $(\widehat{W}_+,\widehat{\cW}_+)$ under which the sequence of first coordinates $(Z_n)_{n \geq 0}$ has the law of a simple random walk on $\bbZ^d$ starting from $x \in \bbZ^d$ and the sequence of second coordinates $(\zeta_n)_{n \geq 0}$ are i.i.d.~exponential variables with parameter $1$, independent from $(Z_n)_{n \geq 0}$. We call $E_x$ the expectation associated to $P_x$.

To $\widehat{w} \in \widehat{W}_+$, we attach a continuous-time trajectory $(X_t(\widehat{w}))_{t \geq 0}$ via the definition
\begin{equation}
X_t(\widehat{w}) = Z_n(\widehat{w}), \mbox{ for } t \geq 0,\quad \text{ when }\, \sum_{i = 0}^{n-1} \zeta_i(\widehat{w}) \leq t < \sum_{i = 0}^{n} \zeta_i(\widehat{w}),
\end{equation}
the left-hand side being $0$ if $n = 0$. Thus, $(X_t)_{t\geq 0}$  under $P_x$ describes the continuous-time simple random walk with unit jump rates starting from $x$. 

For a subset $K \subseteq \mathbb{Z}^d$, we introduce $H_K = \inf \lbrace t \geq 0; X_t \in K \rbrace$, $\widetilde{H}_K = \inf \lbrace t \geq \zeta_1; X_t \in K \rbrace$, and $T_K = \inf \lbrace t \geq 0; X_t \notin K \rbrace$, the entrance, hitting and exit times of $K$. The Green function of the random walk $g(\cdot,\cdot)$ is then defined by
\begin{equation}
\label{eq:GreenFunction}
g(x,y) = E_x\bigg[\int_0^\infty 1_{\lbrace X_s = y \rbrace} ds \bigg], \text{ for } x,y \in \mathbb{Z}^d,
\end{equation}
and since $d \geq 3$, it is finite. Moreover, one has $g(x,y) = g(y,x) = g(x-y,0) =: g(x-y)$ and the following asymptotic behavior (see e.g. Theorem 5.4, p.31 of~\cite{lawler2013intersections}):
\begin{equation}
\label{eq:AsymptoticBehaviourGreen}
g(x) \sim \frac{C_d}{|x|^{d-2}}, \qquad \text{as } |x| \rightarrow \infty, \text{ with }C_d = \frac{d}{2\pi^{\frac{d}{2}}}\Gamma\left(\frac{d}{2}-1 \right).
\end{equation}
For a finitely supported function $f:\bbZ^d\to \bbR$ we write 
\begin{equation}
\label{eq:DefGOperator}
G f(x) = \sum_{y\in \bbZ^d} g(x,y) f(y),\quad x \in \bbZ^d.
\end{equation}
The equilibrium measure of a finite subset $K \subseteq \mathbb{Z}^d$ is defined by
\begin{equation}
\label{eq:EqMeasure}
e_K(x) = P_x[\widetilde{H}_K = \infty] \mathbbm{1}_K(x), \text{ for } x \in \mathbb{Z}^d,
\end{equation}
and its total mass
\begin{equation}
\text{cap}_{\mathbb{Z}^d}(K) = \sum_{x \in K}e_K(x)
\end{equation}
is called the (discrete) capacity of $K$. If $K$ is non-empty, we denote by $\overline{e}_K$ the normalized equilibrium measure of $K$, that is
\begin{equation}
\label{eq:NormEqMeausre}
\overline{e}_K(x) = \frac{e_K(x)}{\capa_{\bbZ^d}(K)},\quad x \in \bbZ^d.
\end{equation}
 Also, for a set $K\subseteq \bbZ^d$ we write 
\begin{equation}
\label{eq:harmPotentialDiscrete}
    h_K(x) = P_x[H_K<\infty], \quad x\in \bbZ^d,
\end{equation}
for the harmonic potential associated to $K$. 
Recall that for finite $K \subseteq \mathbb{Z}^d$, one has
\begin{equation}
\label{eq:EquilibriumPotential}
h_K(x) = G e_K(x), \text{ for }x \in \mathbb{Z}^d,
\end{equation} 
see e.g.\ Theorem 25.1, p.300 of \cite{spitzer2013principles}.  
Finally, for functions $f,g:\bbZ^d \to \bbR$ we define the discrete Dirichlet form by
\begin{equation}
    \label{eq:discreteDF}
    \cE_{\bbZ^d}(f,g) = \frac{1}{4d}\sum_{x\sim y} (f(y)-f(x)) (g(y)-g(x)),
\end{equation}
whenever the above expression is absolutely summable. Moreover, we will use the shorthand notation $\cE_{\bbZ^d}(f) = \cE_{\bbZ^d}(f,f)$.
\vspace{0.5\baselineskip}

We now introduce continuous-time random interlacements. We refer to~\cite{drewitz2014introduction} for more details on (discrete-time) random interlacements and to~\cite{sznitman2012isomorphism} for the case of continuous-time random interlacements. We write $\widehat{W}^\ast$ for the space $\widehat{W}$ modulo time shift, that is $\widehat{W}^\ast = \widehat{W} / \sim$ where $\widehat{w} \sim \widehat{w}'$ if there is a $k \in \bbZ$ such that $\widehat{w} = \widehat{w}'(\cdot + k)$. Moreover, we denote by $\pi^\ast : \widehat{W} \rightarrow \widehat{W}^\ast$ the canonical projection and endow $\widehat{W}^\ast$ with the push-forward $\sigma$-algebra of $\widehat{\cW}$ under $\pi^\ast$. For a finite set $A\subseteq \bbZ^d$ we denote  by $\widehat{W}^\ast_A$ the subset of $\widehat{W}^\ast$ of trajectories modulo time-shift that intersect $A$. For $\widehat{w}^*\in \widehat{W}^\ast_A$ we define $\widehat{w}^*_{A,+}$ to be the unique element of $\widehat{W}_+$ that follows $\widehat{w}^*$ step by step from the first time it enters $A$. More precisely, taking the unique $\widehat{w}\in \widehat{W}$ such that $\pi^\ast(\widehat{w}) = \widehat{w}^*$, $\widehat{w}(0) \in A\times(0,\infty)$ and $\widehat{w}(k) \notin A\times(0,\infty)$ for all $k<0$, we define $\widehat{w}^*_{A,+}(k) = \widehat{w}(k)$ for all $k\geq 0$.

The continuous-time random interlacements is then a Poisson process on $\widehat{W}^\ast\times \bbR_+$, with intensity measure $\widehat{\nu}(\De \widehat{w}^*)\,\De u$, where $\widehat{\nu}$ is a $\sigma$-finite measure on $\widehat{W}^*$ such that its restriction to $\widehat{W}^*_A$ (denoted by $\widehat{\nu}_A$) is equal to $\pi^*\circ \widehat{Q}_A$ where $\widehat{Q}_A$ is a finite measure on $\widehat{W}$ such that if $(X_t)_{t\in \bbR}$ is the continuous-time walk attached to $\widehat{w}\in \widehat{W}$ (see (1.7) in~\cite{sznitman2012isomorphism}), then
\begin{equation}
    \widehat{Q}_A [X_0 = x] = e_A(x),
\end{equation}
and when $e_A(x)>0$,
\begin{equation}
    \begin{minipage}{0.75\linewidth}
        under $\widehat{Q}_A$ conditioned on $X_0 = x$, $(X_t)_{t\geq 0}$ and the right-continuous regularization of $(X_{-t})_{t\geq 0}$ are independent and are distributed respectively as $(X_t)_{t\geq 0}$ under $P_x$ and as $(X_t)_{t\geq 0}$ under $P_x[\,\cdot\,| \widetilde{H}_A = \infty]$.
    \end{minipage}
\end{equation}
The space $\Omega$ where the Poisson point measure is defined can be chosen as
\begin{equation}
    \Omega = \left\{ 
        \begin{minipage}{0.75\linewidth}
            $\omega = \sum_{i\geq 0} \delta_{(\widehat{w}^*_i, u_i)}$; with $\widehat{w}_i^*\in \widehat{W}^*$ for each $i\geq 0$,  $u_i$ distinct, so that $\omega(\widehat{W}^\ast_A\times [0,u])<\infty$ and $\omega(\widehat{W}^\ast_A\times \bbR_+)=\infty$, for any non-empty finite $A\subseteq \bbZ^d$ and $u\geq 0$. \phantom{$\widehat{Q}$}
        \end{minipage}
    \right\}.
\end{equation}
The space $\Omega$ is endowed with the canonical $\sigma$-algebra and we denote by $\bbP$ the law on $\Omega$ under which $\omega$ is a Poisson point process of intensity measure $\widehat{\nu}\otimes \De u$.

Then, given $\omega=\sum_{i\geq 0} \delta_{(\widehat{w}^*_i, u_i)}$ in $\Omega$ and $u\geq 0$, the random interlacement at level $u$, and the vacant set at level $u$, are defined as the random subsets of $\bbZ^d$
\begin{equation}
    \cI^u(\omega) = \bigcup_{i: u_i\leq u} \mathrm{Range}(\widehat{w}^*_i),\quad \cV^u(\omega) = \bbZ^d \setminus \cI^u(\omega),
\end{equation}
where for $\widehat{w}^*\in \widehat{W}^*$, $\mathrm{Range}(\widehat{w}^*)$ stands for the set of points in $\bbZ^d$ visited by the first coordinate sequence associated to any $\widehat{w}\in\widehat{W}$ such that $\pi^*(\widehat{w}) = \widehat{w}^*$.

 
The main object of interest for us is $L_{x,u}(\omega)$, the  (continuous) occupation time at site $x$ and level $u$ of random interlacements, that is, the total time spent at $x$ by all trajectories $\widehat{w}_i^*$ with label $u_i\leq u$ in the cloud $\omega = \sum_{i\geq 0} \delta_{(\widehat{w}_i^*, u_i)} \in \Omega$. Formally, we define
\begin{equation}
\begin{split}
& L_{x,u}(\omega) = \sum_{i \geq 0} \sum_{n\in \bbZ} \zeta_n(\widehat{w}_i) \mathbbm{1}_{\{ Z_n(\widehat{w}_i) = x, u_i \leq u \}}, \text{ for } x\in \bbZ^d, u \geq 0, \\
& \text{for }\omega  = \sum_{i\geq 0} \delta_{(\widehat{w}_i^*, u_i)} \in \Omega, \text{ and } \pi^\ast(\widehat{w}_i) = \widehat{w}_i^\ast \text{ for any } i \geq 0.
\end{split}
\end{equation}
One knows that $\bbE[L_{x,u}] = u$ and also the following formula for the Laplace transform of $(L_{x,u})_{x\in \bbZ^d}$ (see Theorem 2.1 of~\cite{sznitman2012random}). Namely,
for any $V:\bbZ^d\to \bbR$ finitely supported such that $\|G |V| \|_\infty<1$, and $u\geq 0$,\begin{equation}\label{eq:laplacetransform}
            \bbE\Big[\exp\Big\{\sum_{x\in \bbZ^d} V(x) L_{x,u}\Big\}\Big] =  \exp\{ u \langle V, (I-GV)^{-1} 1 \rangle_{\bbZ^d} \}.
    \end{equation} 
   On the right-hand side of this equation, $GV$ stands for the composition of $G$ with the multiplication operator by $V$, so that $GV$ operates in a natural way on $L^\infty(\bbZ^d)$, the space of bounded real functions on $\bbZ^d$ (note that $\|GV \|_{L^\infty(\bbZ^d) \rightarrow L^\infty(\bbZ^d)} = \| G | V| \|_\infty$).
    
Even though~\eqref{eq:laplacetransform} is enough for our purposes, more is known for the logarithm of the Laplace transform and a variational formula is provided in Sections 2 and 4 of~\cite{li2015large}.
\vspace{\baselineskip}

We now introduce Brownian motion on $\bbR^d$ and present some aspects of its potential theory, in a similar fashion as it was done for the simple random walk above. Let $(Z_t)_{t \geq 0}$ be the canonical process on $C(\bbR_+, \bbR^d)$ and denote by $W_z$ the Wiener measure starting from $z \in \bbR^d$ such that under $W_z$, $(Z_t)_{t \geq 0}$ is a Brownian motion starting from $z \in \bbR^d$. For any open or closed set $B\subseteq \bbR^d$, we introduce $H_B = \inf \lbrace s \geq 0; Z_s \in B \rbrace$ and $\widetilde{H}_B = \inf \lbrace s > 0; Z_s \in B \rbrace$, the entrance and hitting times of $B$ for Brownian motion, and $T_B = \inf \lbrace s \geq 0; Z_s \notin B \rbrace ( = H_{B^c})$, the exit time of Brownian motion from $B$. For later use we also define the first time when $Z$ moves at $|\cdot|_\infty$-distance $ r\geq 0$ from its starting point, 
\begin{equation}
\tau_r = \inf \lbrace s \geq 0; |Z_s - Z_0|_\infty \geq r \rbrace.
\end{equation}
For an open or closed set $B \subseteq \bbR^d$, one introduces the harmonic potential of $B$,
\begin{equation}\label{eq:HarmonicPot}
\scrh_B(z) = W_z[\widetilde{H}_B < \infty], \qquad z \in \bbR^d.
\end{equation}

For $f,g  \in H^1(\bbR^d)$, the usual Sobolev space of square-integrable functions on $\bbR^d$ with square-integrable weak derivatives, one defines the Dirichlet form attached to Brownian motion
\begin{equation}
\cE(f) = \frac{1}{2} \int_{\bbR^d} \left|\nabla f(x)\right|^2\De x,
\end{equation}
and by polarization one defines furthermore
\begin{equation}
\cE(f,g) = \frac{1}{4}\left(\cE(f+g) + \cE(f-g) \right).
\end{equation}
Note that $\cE(\cdot,\cdot)$ defined in this way is bilinear and its definition can be extended to the space of all weakly differentiable functions with finite Dirichlet energy.  
Combining Theorem 4.3.3, p.\ 171 of~\cite{fukushima2010dirichlet} with Theorem 2.1.5, p.\ 72 of the same reference, one knows that for any bounded and either open or closed set $B \subseteq \bbR^d$, $\scrh_B$ is in this extended Dirichlet space of $(\cE,H^1(\bbR^d))$ (see Example 1.5.3 in \cite{fukushima2010dirichlet} for a characterization of this space) and it holds that
\begin{equation}
\label{eq:EqualityCapacityDirichlet}
\capa(B) = \cE(\scrh_B, \scrh_B).
\end{equation}
Moreover, if $B$ is open and bounded and $(B_n)_{n \geq 1}$ is a sequence of compact sets such that $B_n \uparrow B$, then (see Proposition 1.13, p.60 of~\cite{port2012brownian})
\begin{equation}\label{eq:increasing_capacity}
\capa(B_n) \uparrow \capa(B).
\end{equation}
We also note here, that if $f \in L^2(\bbR^d)\cap L^1(\bbR^d)$ and $g$ is in the extended Dirichlet space of $(\cE,H^1(\bbR^d))$, one has
\begin{equation}
\label{eq:CSEnergies}
|\langle f,g \rangle|^2 \leq E(f) \cE(g,g),
\end{equation}
where $E(f) = \int f(x) g_{BM}(x,y) f(y)\, \De x\De y$ is the energy associated to the function $f$, with $g_{BM}(x,y)$ being the Green function of the standard Brownian motion on $\bbR^d$. To see this inequality, one can for instance show it first in the case where $f,g$ are smooth and compactly supported, and then use an approximation argument (compare also with Lemma 1.5.3, p.\ 39, and Theorem 1.5.4, p.\ 44, of~\cite{fukushima2010dirichlet}).  \vspace{\baselineskip} 

We now recall an asymptotic lower bound from~\cite{nitzschner2017solidification} on the capacity of `porous interfaces' surrounding $A \subseteq \mathbb{R}^d$ and a related estimate from~\cite{chiarini2018entropic}. These estimates will be pivotal in the derivation of the bound~\eqref{eq:MainResultSection4} of Section~\ref{sec:entropicpush}. Let $U_0$ be a non-empty Borel subset of $\bbR^d$ with complement $U_1 = \mathbb{R}^d \setminus U_0$ and boundary $ \partial U_0 = \partial U_1$. One measures the local density of $U_1$ at $x \in \bbR^d$ in dyadic scales
\begin{equation}
\widehat{\sigma}_{\ell}(x) = \frac{|B_{\infty}(x,2^{-\ell})\cap U_1|}{|B_{\infty}(x,2^{-\ell})|}, \qquad \ell \in \bbZ,
\end{equation}
 where $| \cdot |$ stands for the Lebesgue measure on $\mathbb{R}^d$. We furthermore introduce for $\ell_\ast$ non-negative integer and for a non-empty compact subset $A$ of $\mathbb{R}^d$
 \begin{equation}
 \label{eq:SegmentationClass}
 \mathcal{U}_{\ell_\ast, A} = \, \begin{minipage}{0.7\linewidth}
  the collection of bounded Borel subsets $U_0 \subseteq \mathbb{R}^d$ with $\widehat{\sigma}_{\ell}(x) \leq \tfrac{1}{2}$  for all $x\in A$  and $\ell \geq \ell_\ast$.
 \end{minipage} 
\end{equation}  
For a given non-empty Borel subset $U_0 \subseteq \bbR^d$, $\varepsilon > 0$ and $\eta \in (0,1)$ we consider the following class of `porous interfaces' 
\begin{equation}
\label{eq:ClassofporousInterf}
\begin{split}
\mathscr{S}_{U_0,\varepsilon,\eta} &= \ \text{the class of }\Sigma \subseteq \mathbb{R}^d \text{ compact with } W_z[H_\Sigma < \tau_\varepsilon] \geq \eta, \text{ for } z \in \partial U_0.
\end{split}
\end{equation}
Essentially, $\varepsilon$ controls the distance of the porous interface $\Sigma$ from $\partial U_0$ and $\eta$ corresponds to the strength with which it is `felt'. With this, we can quote the capacity lower bound (3.16) of Corollary 3.4 in~\cite{nitzschner2017solidification}, which provides for $\eta \in (0,1)$ in the limit $\varepsilon/ 2^{-\ell_\ast}$ going to zero the following uniform control:
\begin{equation}
\label{eq:SolidificationEstimate}
\lim_{u \rightarrow 0} \inf_{\varepsilon \leq u2^{-{\ell_\ast}}} \inf_A \inf_{U_0 \in \mathcal{U}_{\ell_\ast,A}} \inf_{\Sigma \in \mathscr{S}_{U_0,\varepsilon,\eta}} \frac{\text{cap}(\Sigma)}{\text{cap}(A)} = 1,
\end{equation}
where $A$ varies in the class of non-empty compact subsets of $\bbR^d$ with positive capacity. Finally, we recall a result from~\cite{chiarini2018entropic} (see Lemma 2.2). It states that in the limit $\varepsilon/2^{-\ell^*}\to 0$, the Dirichlet energy of $\scrh_A-\scrh_{\Sigma}$ is bounded from above by the capacity difference $\capa(\Sigma) -\capa(A)$, uniformly over all compacts $A\subseteq \bbR^d$ and all porous interfaces $\Sigma$. More precisely, for any $\eta\in(0,1)$ fixed
\begin{equation}
\label{eq:SolidificationDirichletEnergy}
    \limsup_{u\to 0} \sup_{\varepsilon \leq u 2^{-\ell ^*}} \sup_A \sup_{U_0 \in \mathcal{U}_{\ell_\ast,A}} \sup_{\Sigma \in \mathscr{S}_{U_0,\varepsilon,\eta}} \bigg[\cE(\scrh_A - \scrh_{\Sigma})-\big(\capa(\Sigma)-\capa(A)\big)\bigg] = 0,
\end{equation}
where $A$ varies in the class of non-empty compact subsets of $\bbR^d$. Similar to~\cite{chiarini2018entropic}, this result will be needed in the proof of Theorem~\ref{thm:MainResult} (see Step 5,~\eqref{eq:BoundSigma_LargeDirichletform}), to rule out, with high probability, the existence of atypical interfaces of bad boxes.

\section{Laplace functional of occupation-time measures of random interlacements}
\label{sec:occupationtimebound}

In this section, we derive an identity for the Laplace functional of the occupation-time measure with respect to a certain class of potentials with a small perturbation. The main result of this section, Lemma \ref{Lemma_Gauge} below, will be instrumental in giving a bound on the probability of a large deviation in the occupation-time field of random interlacements from its expectation.  

These bounds will enter the proof of our main result in the form of Corollary \ref{Cor_OccTime}, and replace in essence certain bounds involving the Borell-TIS inequality in the case of the Gaussian free field, see Proposition 4.3 of \cite{chiarini2018entropic}.

We start with some notation. For a function $V : \bbZ^d \rightarrow \bbR$ vanishing outside a finite subset of $\bbZ^d$, we define the \emph{gauge function} $\gamma_V : \bbZ^d \rightarrow [0,\infty]$ as 

\begin{equation}
\label{eq:GaugeFunction}
\gamma_V(x) = E_x\Big[\exp\Big( \int_0^\infty V(X_s) \De s \Big) \Big],
\end{equation}
where we recall that under $P_x$, $(X_s)_{s \geq 0}$ is a continuous-time random walk on $\bbZ^d$ starting from $x \in \bbZ^d$, with $E_x$ the expectation associated to $P_x$. If $V$ is such that $\| G |V| \|_\infty < 1$, one can show by expanding the exponential and using the strong Markov property in the same way as in 
Proposition 1.2 and Remark 2.1 of \cite{li2015large}, that
\begin{equation}
\label{eq:CentralEqualityGamma}
\gamma_V(x) = (I - GV)^{-1} 1 (x),
\end{equation}
which will essentially allow perturbative calculations of $\gamma_V$ (recall that $GV$ stands for the composition of $G$, defined in \eqref{eq:DefGOperator}, and the multiplication by $V$). Of particular interest for us will be the case where $V$ corresponds to a multiple of the equilibrium measure of a finite, non-empty set $C \subseteq \bbZ^d$, since then $\gamma_{ae_C}$ (where $a \in (0,1)$) will correspond to a multiple of the equilibrium potential of the set $C$, shifted by one (see Remark \ref{Remark_CalculationGamma}). 

We now present the main result of this section, in which we develop certain perturbation formulae for $\gamma_V$ which may be of independent interest.
\begin{lemma}
\label{Lemma_Gauge}
Let $V, V' : \bbZ^d \rightarrow \bbR$ be two functions which are zero outside of a non-empty finite set, and assume that
\begin{equation}
\label{eq:ConditionsGV}
\| G |V| \|_\infty < 1, \qquad  \| G |V'| \|_\infty < 1.
\end{equation}
Then, the following perturbation formulae hold:
\begin{equation}
\label{eq:PerturbationFormula}
\gamma_{V'} - \gamma_V = (I - GV)^{-1}G (V' - V) \gamma_{V'},
\end{equation}
together with
\begin{equation}
\label{eq:PerturbationTestedAgainstV}
\langle V', \gamma_{V'} \rangle_{\bbZ^d} - \langle V, \gamma_V \rangle_{\bbZ^d} = \langle (V' - V) \gamma_V, \gamma_{V'} \rangle_{\bbZ^d}.
\end{equation}
Moreover, if we assume additionally that 
\begin{equation}
\label{eq:ConditionsGV2}
 \| (I - G|V|)^{-1} G |V' - V| \|_{\infty} < 1,
\end{equation}
then it holds that
\begin{equation}
\label{eq:PerturbationExplicit}
\gamma_{V'} = ( I -  (I- GV)^{-1}G(V' - V)  )^{-1} \gamma_V.
\end{equation}
\end{lemma}
\begin{proof}
We view $GV$ and $GV'$ as operators acting on $L^\infty(\bbZ^d)$. By \eqref{eq:ConditionsGV}, the resolvent sets of both operators contain $1$ and by the second resolvent identity (see Lemma 6.5 in~\cite{teschl2009mathematical}) we have
\begin{equation}
(I - GV')^{-1} - (I - GV)^{-1} = (I-GV)^{-1} G(V' - V)  (I - GV')^{-1}.
\end{equation}
Applying this operator equation to the constant function $1$ and using that \eqref{eq:CentralEqualityGamma} holds for both $V$ and $V'$ readily implies \eqref{eq:PerturbationFormula}. 
Next, we will prove \eqref{eq:PerturbationTestedAgainstV}. Upon multiplication of \eqref{eq:PerturbationFormula} with $V$ and summation over $x \in \bbZ^d$, we see that
\begin{equation}
\langle V, \gamma_V' \rangle_{\bbZ^d} - \langle V, \gamma_V \rangle_{\bbZ^d} = \langle V, (I-GV)^{-1} G(V' - V) \gamma_{V'} \rangle_{\bbZ^d}.
\end{equation}
 Since $\| G|V| \|_\infty < 1$, we can expand the sum $(I - GV)^{-1}$ into a series and rewrite the last expression as
 \begin{equation}
 \begin{split}
 \sum_{n = 0}^\infty \langle 1, V (GV)^n G (V' &- V)\gamma_{V'} \rangle_{\bbZ^d}  = \sum_{n = 1}^\infty \langle 1, (VG)^n (V' - V) \gamma_{V'} \rangle_{\bbZ^d} \\
 & = \langle ((I - GV)^{-1} - I)1, (V' - V) \gamma_{V'} \rangle_{\bbZ^d} \\
 & = \langle \gamma_V - 1, (V' - V) \gamma_{V'} \rangle_{\bbZ^d}.
 \end{split} 
 \end{equation}
 where we used that both $V$ and $G$ are symmetric operators.
 The claim follows easily by rearranging the terms. We finally turn to the proof of \eqref{eq:PerturbationExplicit}. From the perturbation identity \eqref{eq:PerturbationFormula}, we can conclude that
\begin{equation}
(I  - (I - GV)^{-1} G (V' - V) ) \gamma_{V'} = \gamma_V.
\end{equation}
If \eqref{eq:ConditionsGV2} holds, the operator acting on $\gamma_{V'}$ in the above equation has a bounded inverse $(I - (I - GV)^{-1} G (V' - V) )^{-1}$ and therefore, \eqref{eq:PerturbationFormula} follows. 

\end{proof}

In our main application, the perturbation of a potential $V$ is of a certain size $\delta > 0$, and it will be of interest to control deviations of the occupation-time profile from its expectation in terms of powers in $\delta$. The following corollary will be helpful in the proof of the main Theorem \ref{thm:MainResult} (more precisely in Proposition~\ref{PropOccTimeBounds}) of this article. In essence, it follows from combining the result of Lemma \ref{Lemma_Gauge} with a well-known formula for the Laplace functional of the occupation-time measure of random interlacements. 
\begin{corollary}
\label{Cor_OccTime}
Let $V, \eta : \bbZ^d \rightarrow \bbR$ be functions vanishing outside a finite set and $\delta > 0$ such that, with $V' = V + \delta \eta$, \eqref{eq:ConditionsGV} and \eqref{eq:ConditionsGV2} are both fulfilled. Then, for any $t\in\bbR$, 
\begin{equation}
\label{eq:ClaimCorollary}
\bbP\Big[\langle \cL_{u}, V'\rangle_{\bbZ^d} \geq u\langle V', \gamma_V^2 \rangle_{\bbZ^d} + t\Big] \leq \exp\Big(- u \cE_{\bbZ^d}(\gamma_V -1) - t + u\cR_{\delta,\eta,V} \Big),
\end{equation}
where $\cL_{u} := \sum_{x \in \bbZ^d} L_{x,u} \delta_x = \scrL_{1,u}$ and 
\begin{equation}
\cR_{\delta,\eta,V} = \langle|\eta|,1\rangle_{\bbZ^d} \| \gamma_V \|_\infty^2  \frac{\delta^2 \| (I - G|V|)^{-1} G |\eta| \|_\infty }{ 1 - \delta \| (I - G|V|)^{-1} G |\eta| \|_\infty}.
\end{equation}
\end{corollary}
\begin{proof}
By the exponential Markov inequality, one has
\begin{equation}
\label{eq:MainBoundLemma}
\begin{split}
\bbP[\langle & \cL_{u}, V ' \rangle_{\bbZ^d}  \geq u\langle V', \gamma_V^2  \rangle_{\bbZ^d} + t] \\
& \leq \exp\big( - u\langle V', \gamma_V^2  \rangle_{\bbZ^d} - t \big) \bbE\big[ \exp( \langle \cL_{u}, V'\rangle_{\bbZ^d} ) \big].
\end{split}
\end{equation}
By~\eqref{eq:laplacetransform} (see also Theorem 2.1 of \cite{sznitman2012random}) and since we assumed  $\| G |V'|  \|_\infty < 1$, the expectation can be written as
\begin{equation}
\bbE\big[ \exp( \langle \cL_{u}, V'\rangle_{\bbZ^d} ) \big] = \exp \Big( u \langle V', (I - GV')^{-1} 1 \rangle_{\bbZ^d} \Big) = \exp\big( u \langle V', \gamma_{V'}\rangle_{\bbZ^d} \big). 
\end{equation}
Now, we see that since $V' - V = \delta \eta$, 
\begin{equation}
\begin{split}
\langle V', \gamma_{V'}\rangle_{\bbZ^d} & \stackrel{\eqref{eq:PerturbationTestedAgainstV}}{=} \langle V, \gamma_{V} \rangle_{\bbZ^d} + \delta \langle \eta \gamma_V, \gamma_{V'} \rangle_{\bbZ^d}.
\end{split}
\end{equation}

Since the assumption \eqref{eq:ConditionsGV2} is fulfilled, we can insert \eqref{eq:PerturbationExplicit} into the above equation
\begin{equation}
\label{eq:VGammaV_prime}
\langle V', \gamma_{V'} \rangle_{\bbZ^d} = \langle V, \gamma_V \rangle_{\bbZ^d} + \delta \langle \eta \gamma_V, (I - \delta(I-GV)^{-1} G \eta)^{-1} \gamma_V \rangle_{\bbZ^d}.
\end{equation}

By construction, we have a natural ordering in terms of powers in $\delta$ of the right-hand side, and using again \eqref{eq:ConditionsGV2}, we arrive at 
\begin{equation}
\langle V', \gamma_{V'} \rangle_{\bbZ^d} = \langle V, \gamma_V \rangle_{\bbZ^d} + \delta \langle \eta, \gamma_V^2 \rangle_{\bbZ^d} + \delta \sum_{n \geq 1} \langle \eta \gamma_V, (\delta (I - GV)^{-1} G\eta)^n \gamma_V \rangle_{\bbZ^d}.
\end{equation}
The absolute value of the sum can be bounded as follows:
\begin{equation}
\label{eq:HigherOrderTerms}
\begin{split}
\bigg\vert  \sum_{n \geq 1} \langle \eta \gamma_V, & (\delta (I - GV)^{-1} G\eta)^n \gamma_V \rangle_{\bbZ^d} \bigg\vert \\
 & \leq \sum_{n \geq 1} \langle| \eta |,1\rangle_{\bbZ^d} \| \gamma_V \|_\infty^2  \Big(\delta \| (I - G|V|)^{-1} G |\eta| \|_\infty\Big)^n \\
& =  \langle| \eta |,1\rangle_{\bbZ^d} \| \gamma_V \|_\infty^2  \frac{\delta \| (I - G|V|)^{-1} G |\eta| \|_\infty }{ 1 - \delta \| (I - G|V|)^{-1} G| \eta| \|_\infty}.
\end{split}
\end{equation}
Collecting \eqref{eq:MainBoundLemma}--\eqref{eq:HigherOrderTerms}, we arrive at
\begin{equation}
\label{eq:penultimate_step}
\bbP[\langle \cL_{u}, V ' \rangle_{\bbZ^d} \geq u\langle V', \gamma_V^2  \rangle_{\bbZ^d} + t] \leq \exp\bigg(- u \langle V, \gamma_V^2 - \gamma_V \rangle_{\bbZ^d} - t + u\cR_{\delta,\eta,V}  \bigg).
\end{equation}
To conclude, we are left with showing that the scalar product can be rewritten as a Dirichlet form. To do this, we note that $(\Delta + V)\gamma_V = 0$ (with $\Delta$ the discrete Laplacian), as can be shown explicitly by using the definition of $\gamma_V$ \eqref{eq:GaugeFunction} and expanding $(I-GV)^{-1}$. For $V$ not zero everywhere, $\gamma_V > 0$, one therefore has $V = - \frac{\Delta \gamma_V}{\gamma_V}$ and thus
\begin{equation}
\label{eq:DirichletFormGammaV}
\begin{split}
\langle V, \gamma_V^2 \rangle_{\bbZ^d} - \langle V, \gamma_V \rangle_{\bbZ^d} & = \langle - \Delta \gamma_V, \gamma_V \rangle_{\bbZ^d} - \langle - \Delta \gamma_V, 1 \rangle_{\bbZ^d} \\
& =  \langle - \Delta \gamma_V, \gamma_V - 1 \rangle_{\bbZ^d} \\
& = \langle - \Delta (\gamma_V - 1), \gamma_V - 1 \rangle_{\bbZ^d} \\
& = \cE_{\bbZ^d}(\gamma_V - 1, \gamma_V - 1). 
\end{split}
\end{equation}
In the last step, we used the summation by parts formula. The formula \eqref{eq:DirichletFormGammaV} also holds trivially if $V$ is identical to zero. By inserting \eqref{eq:DirichletFormGammaV} into \eqref{eq:penultimate_step}, the claim of the Corollary follows.
\end{proof}

\begin{remark}
\label{Remark_CalculationGamma}

\begin{enumerate}
\item For the application that we have in mind (cf.\ Proposition~\ref{PropOccTimeBounds}), it will be crucial that the remainder term $\cR_{\delta,\eta, V}$ is of order $\delta^2$ as $\delta \rightarrow 0$ if $\|(I-G|V|)^{-1} G |\eta| \|_\infty$ and $\| \gamma_V \|_\infty$ stay bounded away from zero and infinity over the class of possible $\eta$ and $V$ that we are interested in. In fact,  the main contribution in terms of $\delta$ in the exponential in~\eqref{eq:ClaimCorollary} will come in the term $t$ and will be linear in $\delta$.

\item In the situation $V = ae_C$ with $C \subseteq \bbZ^d$ finite and non-empty and $a \in (0,1)$, one has $\gamma_{ae_C}(x) = 1+ \frac{a}{1-a} h_C(x)$. To see this, we remark that $G|V| = a h_C < 1$, thus we have
\begin{equation}
\begin{split}
\gamma_{ae_C}(x) & = (1 - GV)^{-1}1 (x)  = \sum_{n = 0}^\infty a^n (Ge_C)^n 1(x) \\
& = 1 + \sum_{n = 1}^\infty a^n(Ge_C)^{n-1} \underbrace{(G e_C) (x)}_{ = h_C(x) ( = 1 \text{ on }C)} = 1 + \frac{a}{1-a}h_C(x).
\end{split}
\end{equation}
\end{enumerate}
\end{remark}

\section{Entropic push of the occupation-time field by disconnection}\label{sec:entropicpush}

In this section we prove our main result, namely Theorem~\ref{thm:MainResult}. It states an asymptotic upper bound on the joint occurrence of the \emph{disconnection event} $\cD^u_N$ and the event that, for a fixed $R>0$ so that $[-M,M]^d\subseteq  B_R$, the $d_R$-distance (see~\eqref{Distance} below)
between the random measure $\scrL_{N,u}$ (the scaled occupation-time measure of random interlacements) and $\scrM^u_{\mathring{A}}(x)\De x$  becomes large. 
Informally, $d_R$ is a metric that measures the distance between two non-negative measures on $B_R$ by comparing on one hand the 1-Wasserstein distance between their versions normalized to one and on the other hand their total masses. 

If the critical levels $\overline{u}$, $u_\ast$ and $u_{\ast\ast}$ all coincide and if $A$ is regular in the sense that $\capa(\mathring{A}) = \capa(A)$, we furthermore obtain Corollary \ref{thm:CorollaryPush}, which can be roughly interpreted as follows: given disconnection, the occupation-time field of random interlacements is pinned with high probability around a local level equal to $\scrM^u_{A}(\tfrac{x}{N}) = (\sqrt{u} + (\sqrt{\overline{u}} - \sqrt{u}) \scrh_A(\tfrac{x}{N}))^2$. These results and the methods used in the proof are similar in spirit to corresponding ones in the case of level-set percolation of the Gaussian free field (see Section 4 of~\cite{chiarini2018entropic}).
\vspace{\baselineskip}

Before stating the main Theorem, we recall the notion of the $1$-Wasserstein distance and define precisely the metric $d_R$. For $\emptyset \neq J \subseteq \bbR^d$, and a function $\eta: J \rightarrow \bbR$, we denote the Lipschitz constant by
\begin{equation}
\lip(\eta) = \sup_{x,y \in J, x \neq y} \frac{|\eta(x) - \eta(y)|}{|x - y|} \in [0, \infty].
\end{equation}
 Moreover, we define the function space 
\begin{equation}
\label{eq:FunctionClass}
\text{Lip}_1(J) = \big\lbrace \eta : J \rightarrow \bbR ; \ \lip(\eta) \leq 1 \big\rbrace.
\end{equation} 
The $1$-Wasserstein distance (also known under the name of Kantorovich--Rubinstein distance) between \emph{probability measures} $P,Q$ on $J$ is  defined by 
\begin{equation}
    \label{Wasserstein}
    d_{W,J}(P,Q):= \sup_{\eta\in \lip_1(J)} \bigg\{ \int \eta\, \De P- \int \eta\, \De Q\bigg\}.
\end{equation}
It is known that if $J$ is compact, $d_{W,J}$ metrizes weak convergence on the space of probability measures on $J$ (see for example~\cite{villani2008optimal}, Theorem 6.9). 

Our goal is to compare the non-negative measures $\scrL_{N,u}$ and $\scrM_{\mathring{A}}^u(x)\De x$ which in general do not have finite mass on $\bbR^d$. Therefore, we will restrict these measures to arbitrary large boxes $B_R = [-R,R]^d$, with side lengths $2R>0$, and compare both their masses on $B_R$ and their normalized versions on $B_R$ (the choice of a box, i.e.~a ball in sup-norm, instead of a ball in Euclidean norm is not essential, but it will simplify certain geometrical arguments later). More precisely, fix $R>0$ and define
 for non-negative Radon measures $\mu,\nu$ on $B_R$ the distance
\begin{equation}
    \label{Distance}
    d_R(\mu,\nu):= \begin{cases}
        |\mu(B_R)-\nu(B_R)| + d_{W,B_R}\Big(\tfrac{\mu}{\mu(B_R)},\tfrac{\nu}{\nu(B_R)}\Big), & \mbox{if $\mu,\nu\neq 0$,}\\
       0, & \mbox{if $\mu=\nu=0$,}\\
       \infty, & \mbox{otherwise.}
    \end{cases} 
\end{equation}
For simplicity, we write $d_R(\mu,f)$ instead of $d_R(\mu,\nu)$ if $f$ is the Lebesgue density of $\nu$. We are now ready to state the main result of this article. Recall that, according to our convention at the end of Section~\ref{sec:introduction}, all numbered constants are assumed to be positive.

\begin{theorem}\label{thm:MainResult}
Let $\Delta > 0$, $u \in (0,\overline{u})$ and $R>0$ so that $[-M,M]^d\subseteq B_R$. Then,
\begin{equation}
\begin{split}
\label{eq:MainResultSection4}
\limsup_{N } \frac{1}{N^{d-2}} \log \bbP \bigg[ & d_R(\scrL_{N,u}, \scrM^u_{\mathring{A}}) \geq \Delta ; \cD^u_N \bigg] \\
& \leq - \frac{1}{d} (\sqrt{\overline{u}} - \sqrt{u})^2 \capa(\mathring{A}) - c_1(\Delta, R, A,u).
\end{split}
\end{equation}
Moreover, as $u \rightarrow 0$ one has $c_1(\Delta, R, A,u) \sim c_2(\Delta, R, A)\sqrt{u}$.
\end{theorem}
This result should be compared to Theorem 4.1 of \cite{chiarini2018entropic}. In particular, note that the same explanation as around (4.5) of the same reference assures the measurability of the event under the probability in \eqref{eq:MainResultSection4}. The following Corollary gives the interpretation of an `entropic push' alluded to above in the case that the critical levels  $\overline{u}$, $u_\ast$ and $u_{\ast\ast}$ coincide (see Section~\ref{sec:introduction} and the references therein for the definition of these levels). 

\begin{corollary}\label{thm:CorollaryPush}
Let $\Delta, u, R$ be as in Theorem \ref{thm:MainResult} and suppose that $A$ is regular in the sense that $\capa(\mathring{A}) = \capa(A)$. If the critical levels $\overline{u}$, $u_\ast$ and $u_{\ast\ast}$ coincide, one has 
\begin{equation}
\limsup_{N } \frac{1}{N^{d-2}} \log \bbP\bigg[d_R(\scrL_{N,u}, \scrM^u_{A}) \geq \Delta | \cD^u_N\bigg] \leq -c_1(\Delta,R,A,u). 
\end{equation}
\end{corollary}
\begin{proof}
First, we remark that $\capa(\mathring{A}) = \capa(A)$ implies that $\scrh_{\mathring{A}} = \scrh_A$ Lebesgue-a.e., see e.g.~below (3.3) of \cite{chiarini2018entropic}. Therefore the measures  $\scrM^u_{\mathring{A}}(x)\De x$ and $\scrM^u_{A}(x)\De x$ coincide,  and it holds that
\begin{equation}
\begin{split}
&\limsup_{N } \frac{1}{N^{d-2}} \log \bbP\bigg[d_R(\scrL_{N,u}, \scrM^u_{A}) \geq \Delta | \cD^u_N\bigg]  \\ 
&  \leq \limsup_{N } \frac{1}{N^{d-2}} \log \bbP \bigg[ d_R(\scrL_{N,u}, \scrM^u_{\mathring{A}}) \geq \Delta ; \cD^u_N \bigg] - \liminf_{N } \frac{1}{N^{d-2}} \log \bbP[\cD^u_N].
\end{split}
\end{equation}
Combining \eqref{eq:MainResultSection4} with the lower bound on the disconnection probability \eqref{eq:LowerBoundInterlacements} readily proves the claim. 
\end{proof}

The distance $d_R$ is not the only natural metric on the space of non-negative Radon measures on $B_R$. An alternative choice is given for example by the \emph{bounded Lipschitz distance} on $B_R$, which is defined as
\begin{equation}
    \label{bLDistance}
    d_{BL,R}(\mu,\nu) := \sup_{\substack{\|\eta\|_\infty\leq 1\\ \eta\in \lip_1(B_R)}}  \bigg\{ \int \eta\, \De \mu- \int \eta\, \De \nu\bigg\},
\end{equation} 
where $\mu,\nu$ are non-negative Radon measures on $B_R$. In fact, in the proof of Theorem~\ref{thm:MainResult} we will show~\eqref{eq:MainResultSection4} 
with $d_R$ replaced by $d_{BL,R}$ and no restriction on $R>0$. Theorem~\ref{thm:MainResult} is then deduced via Lemma~\ref{comparisonDistances} below and the fact that when $[-M,M]^d\subseteq B_R$ one has that, on the event $\cD^u_N$, both $\scrL_{N,u}$ and $\scrM_{\mathring{A}}^u(x) \De x$ are positive measures.
\begin{lemma}\label{comparisonDistances} Fix $R>0$ and $\mu,\nu$ positive Radon measures on $B_R$. Then
       \begin{equation}
           \label{eq:comparisonDistances}
           \tfrac{1}{ \mu(B_R)\wedge \nu(B_R) +1} \, d_{BL,R}(\mu,\nu) \leq d_R(\mu,\nu) \leq d_{BL,R} \, (\mu,\nu)  \Big(1 + 2\tfrac{(\sqrt{d}R)\vee 1}{\mu(B_R)\vee\nu(B_R)}\Big).
       \end{equation}
   \end{lemma}
   \begin{proof}  We start with the simple observation
   \begin{equation}
       \label{eq:totalmasses}
       |\mu(B_R)-\nu(B_R)| \leq \sup_{\substack{\|\eta\|_{\infty}\leq 1\\ \eta\in \lip_1(B_R)}}  \Big\{\int_{B_R} \eta\,\De \mu - \int_{B_R} \eta\,\De \nu \Big\},
   \end{equation}
   by considering $\eta \equiv 1$ or $\eta\equiv-1$ on $B_R$. By replacing $\eta$ with $\widehat{\eta}=\eta-\eta(0)$ and observing that for $x \in B_R$, $|\widehat{\eta}(x)|\leq (\sqrt{d}R)\vee 1$ by Lipschitz continuity, we get
   \begin{equation}\label{wass}
       \begin{split}
       & d_{W, B_R}\big(\tfrac{\mu}{\mu(B_R)},\tfrac{\nu}{\nu(B_R)}\big) = \sup_{\eta \in \lip_1(B_R)} \Big\{\frac{1}{\mu(B_R)}\int_{B_R} \eta\,\De \mu - \frac{1}{\nu(B_R)}\int_{B_R} \eta\,\De \nu \Big\}\\
       &\leq \sup_{\substack{\|\widehat{\eta}\|_\infty\leq (\sqrt{d}R)\vee1\\ \widehat{\eta} \in \lip_1(B_R)}} \Big\{\frac{1}{\mu(B_R)}\int_{B_R} \widehat{\eta}\,\De \mu - \frac{1}{\nu(B_R)}\int_{B_R} \widehat{\eta}\,\De \nu \Big\}\\
       &\leq \frac{(\sqrt{d}R)\vee 1}{\mu(B_R)\vee \nu(B_R)}\Big( |\mu(B_R)-\nu(B_R)|+\sup_{\substack{\|\eta\|_{\infty}\leq 1\\ \eta\in \lip_1(B_R)}}  \Big\{\int_{B_R} \eta\,\De \mu - \int_{B_R} \eta\,\De \nu \Big\}\Big),
       \end{split}
   \end{equation}
    where in the last inequality we replaced $\widehat{\eta}$ by $\widehat{\eta}/((\sqrt{d}R)\vee 1)$.
       By combining~\eqref{eq:totalmasses} and~\eqref{wass} the second inequality of~\eqref{eq:comparisonDistances} follows by the definition of $d_R$. The first inequality, follows from 
       \begin{equation}
       \begin{split}
           d_{BL,R}(\mu,\nu) & \leq \mu(B_R)\wedge \nu(B_R)\sup_{\substack{\|\eta\|_{\infty}\leq 1\\ \eta\in \lip_1(B_R)}}  \Big\{\tfrac{1}{\mu(B_R)}\int_{B_R} \eta\,\De \mu - \tfrac{1}{\nu(B_R)}\int_{B_R} \eta\,\De \nu \Big\} \\
           & + |\mu(B_R)-\nu(B_R)|
           \end{split}
       \end{equation}
       and the definition of $d_R$.
       \end{proof}

The proof of the main Theorem will rely to a large extent on a certain coarse-graining procedure introduced in \cite{nitzschner2017solidification} that brings into play a class of `porous interfaces' in the sense of \eqref{eq:ClassofporousInterf}. We will therefore recall the relevant scales that play a role in this procedure and provide the necessary definitions that will enter the coarse-graining scheme. 

For $0 <u < \overline{u}$, we consider $\alpha > \beta > \gamma$ in $(u, \overline{u})$ and take a sequence $(\gamma_N)_{N \geq 1}$ of numbers in $(0,1]$ that satisfy the conditions (4.18) of \cite{nitzschner2017solidification}, in particular $\gamma_N \rightarrow 0$ as $N \rightarrow \infty$. 

Next, we consider scales 
\begin{equation}
\label{eq:ScalesL0}
L_0 = \lfloor (\gamma_N N \log N )^{\frac{1}{d-1}} \rfloor,\quad \widehat{L}_0 = 100 d \lfloor \sqrt{\gamma_N} N\rfloor,
\end{equation}
in particular (cf.~(4.24) of \cite{nitzschner2017solidification}) one has that $\widehat{L}_0 / L_0 \rightarrow \infty$ as $N \rightarrow \infty$. Moreover, we will use the lattices 
\begin{equation}
\bbL_0 = L_0 \bbZ^d,\quad \widehat{\bbL}_0 = \tfrac{1}{100d} \widehat{L}_0 \bbZ^d
\end{equation}
and consider for $z \in \bbL_0$ the boxes
\begin{equation}
\begin{split}
B_z & = z + [0,L_0)^d \cap \bbZ^d \subseteq D_z = z + [-3L_0,4L_0)^d \cap \bbZ^d \\
& \subseteq U_z = z + [-KL_0 + 1, KL_0 - 1)^d \cap \bbZ^d,
\end{split}
\end{equation}
where $K \geq c(\alpha, \gamma,\beta, \widetilde{\varepsilon}) (\geq 100)$ is a large integer and the constant corresponds to $c_4(\alpha,\beta,\gamma) \vee c_5(\widetilde{\varepsilon}) \vee c_8(\alpha,\beta,\gamma)$ in the notation of Theorem 2.3, Proposition 3.1 and Theorem 5.1 of \cite{sznitman2017disconnection} which will be sent to infinity eventually.

We denote by $N_u(D_z)$ the number of excursions from $D_z$ to the exterior boundary $\partial U_z$ of $U_z$ that are in the trajectories of the random interlacements up to level $u$, see (3.14) and (2.42) of~\cite{sznitman2017disconnection}. Moreover, we will need the notion of a good$(\alpha,\beta,\gamma)$ box $B_z$ (which is otherwise called bad$(\alpha,\beta,\gamma)$), see (3.11)--(3.13) of~\cite{sznitman2017disconnection}. Roughly speaking, one considers the excursions of the interlacements between $D_z$ and the complement of $U_z$ according to some natural ordering. Being good$(\alpha,\beta,\gamma)$ then corresponds to the existence of a connected set with $| \cdot |_\infty$-diameter exceeding $L_0/10$ in the complement of the first $\alpha \capa_{\bbZ^d}(D_z)$ excursions inside $B_z$, which must be connected to similar components in neighboring boxes $B_{z'}$ in $D_z$ avoiding the first $\beta \capa_{\bbZ^d}(D_z)$ excursions. Additionally, we need that the first $\beta \capa_{\bbZ^d}(D_z)$ excursions have to spend a significant local time of at least $\gamma \capa_{\bbZ^d}(D_z)$ on the inner boundary of $D_z$.

These above properties of boxes have their equivalent in the study of level-set percolation of the Gaussian free field: In our setting, being good$(\alpha,\beta,\gamma)$ roughly plays the same role as being $\psi$-good as defined in Section 5 of \cite{sznitman2015disconnection} does for the Gaussian free field, whereas $N_u(D_z) \geq \beta \capa_{\bbZ^d}(D_z)$ corresponds in essence to the notion of $B_z$ being $h$-good, see (5.9) of \cite{sznitman2015disconnection}. On an informal level, one can understand the variables $N_u(D_z)$ as a means to track a global structure of the occupation-time field that governs the decay of correlations at leading order, while being good$(\alpha,\beta,\gamma)$ only depends on local fluctuations of the field, and for boxes sufficiently far apart, here one has good decoupling properties. 

\paragraph{Outline of the proof} Since the proof of Theorem \ref{thm:MainResult} is done in a multi-step procedure, we will now turn to a detailed description of the outcome of each of the five steps. 
\vspace{0.3\baselineskip}

\noindent \textbf{1.} \emph{Reduction of the uniformity to a location family}: The major effort in this step takes place in Proposition~\ref{handling_supremum}, where we introduce a location family $\{ \chi_\epsilon(\cdot - \tfrac{x}{N}) \}_{x \in B_{R,N}}$ with $B_{R,N} := (NB_R) \cap \bbZ^d$ and consider a discrete convolution of $\eta \in \lip_1(B_R)$ with $\chi_\epsilon(\cdot - \tfrac{x}{N})$ (see \eqref{eq:convolution}), to reduce the set of test functions to a much smaller class. By choosing $\epsilon$ sufficiently small, we obtain an upper bound (see~\eqref{eq:uniformity}) on the probability on the left-hand side of \eqref{eq:MainResultSection4} with $d_R$ replaced by $d_{BL,R}$ in terms of the probability that disconnection occurs and the measures $\scrL_{N,u}$ and $\scrM^u_{\mathring{A}}(x)\De x$ deviate from each other, when tested against elements of the location family (and not against all of $\lip_1(B_R)$). In doing so, we utilize the representation of the Laplace functional of the measure $\scrL_{N,u}$ from \eqref{eq:laplacetransform}.
\vspace{0.3\baselineskip}

\noindent\textbf{2.} \emph{Coarse graining of the disconnection event}: After discarding a `bad event' $\cB_N$ with negligible probability at the relevant order, the effective disconnection event $\widetilde{\cD}^u_N = \cD^u_N \setminus \cB_N$ is decomposed into sub-events $\cD_{N,\kappa}$, where $\kappa \in \cK_N$. A choice of $\kappa \in \cK_N$ will essentially correspond to a set of $L_0$-boxes between $A_N$ and $S_N$ which are all $\text{good}(\alpha,\beta,\gamma)$ and fulfill $N_u(D_z) \geq \beta \capa_{\bbZ^d}(D_z)$. Importantly, this coarse graining is of a `small combinatorial complexity', which means that $|\cK_N| = \exp \{o(N^{d-2}) \}$. Therefore, a union bound will allow us to further reduce the goal of bounding the probability on the right-hand side of \eqref{eq:uniformity} to finding a bound on the probability of the event $\cD_{N,\kappa} \cap \{ |\langle \scrL_{N,u} - \scrM^u_{\mathring{A}}, \chi_\epsilon(\cdot - \tfrac{x}{N}) \rangle | \geq \Delta' \}$ which is uniform in $\kappa \in \cK_N$ \emph{and} in $x \in B_{R,N}$. For a fixed $\kappa \in \cK_N$, upper bounds on the probability of such an event will eventually bring into play the Brownian capacity of a set $\Sigma$ (depending on this $\kappa$), which constitutes a porous interface of boxes surrounding $A' \subseteq \mathring{A}$. It will also be necessary to distinguish two types of $\kappa \in \cK_N$, those for which the Dirichlet energy of $\scrh_{\mathring{A}} - \scrh_{\Sigma}$ is smaller than a given $\mu > 0$, a case that we denote as $\kappa\in\cK^\mu_N$, and those for which the opposite holds. Importantly, in the latter case one can directly use the solidification result \eqref{eq:SolidificationDirichletEnergy} to infer a bound on the probability $\bbP[\cD_{N,\kappa}]$ that is sufficiently good for our purposes. The main result of this step will be \eqref{eq:coarseGraining}, and in what follows we only need to focus on the cases where $\kappa \in \cK^\mu_N$. 
\vspace{0.3\baselineskip}

\noindent\textbf{3.} \emph{Uniform replacement of $\scrM^u_{\mathring{A}}$ by $\cM^u_C$:}  In this step, we aim at replacing the measure $\scrM^u_{\mathring{A}}(x) \De x$ by the measure $\tfrac{1}{N^d} \sum_{x\in \bbZ^d} \cM^u_C(x) \delta_{x/N}$, when tested against a function $\eta \in \lip_1(B_R)$, where $C$ is the set of discrete boxes associated to $\kappa \in \cK_N^\mu$ and $\cM^u_C$ is defined in~\eqref{eq:DiscreteProfileBoxes}. This is the aim of Proposition~\ref{ReplacementProp}. By the result from the previous step, in order to make use of such a replacement, the bound on the error needs to be uniform in $\kappa \in \cK_N^\mu$. Remarkably, this step is purely deterministic and also does not use any solidification estimates. Instead, we entirely rely on a strong coupling technique going back to~\cite{einmahl1989extensions} in the spirit of Komlos, Major and Tusnady to compare $h_C$ with its continuous counterpart $\scrh_{\Sigma}$ and on gradient estimates for bounded harmonic functions.
\vspace{0.3\baselineskip}

\noindent\textbf{4.} \emph{Occupation-time bounds}: After combining the results of the previous steps, we are left with  providing an upper bound on the probability of the intersection of $\cD_{N,\kappa}$ and $\Big\{ \Big|\langle \scrL_{N,u},\chi_\epsilon(\cdot - \frac{x}{N}) \rangle - \tfrac{1}{N^d} \sum_{y\in \bbZ^d} \cM^u_C(y) \chi_\epsilon(\tfrac{y-x}{N})\Big| \geq   \Delta' \Big\}$, which is both uniform in $x \in B_{R,N}$ and in $\kappa \in \cK^\mu_N$. This is done in Proposition~\ref{PropOccTimeBounds}:
The main observation is that the intersection of these two events entails the occurrence of a large deviation of a certain perturbed potential from its expectation, see \eqref{eq:encoded}, which brings us to a situation reminiscent of (4.8) of Theorem 4.2 in~\cite{sznitman2017disconnection}. At this point, the perturbation formulae for the Laplace transform of the occupation time measure (more specifically Corollary \ref{Cor_OccTime}) will be brought into play to yield an exponential bound involving $\capa_{\bbZ^d}(C)$. 
\vspace{0.3\baselineskip}

\noindent\textbf{5.} \emph{Application of the solidification estimates}: 
Via Lemma~\ref{comparisonDistances} we reduce~\eqref{eq:MainResultSection4} to a similar statement for the bounded Lipschitz distance $d_{BL,R}$. Following this, we proceed to collect the bounds obtained in Propositions~\ref{handling_supremum},~\ref{ReplacementProp} and~\ref{PropOccTimeBounds} and use solidification estimates to finalize the proof. 

\paragraph{\textbf{Step 1. Reduction of the uniformity to a location family}} We start by reducing the problem of controlling the supremum of $|\langle \scrL_{N,u} - \scrM^u_{\mathring{A}}, \eta\rangle|$ over the class $\lip_1(B_R)$ to a much smaller class, namely the `location family' $\{\chi_{\epsilon}(\cdot-x/N)\}_{x\in B_{R,N}}$ (recall that $B_{R,N} = (NB_R) \cap \bbZ^d$).

To do this, we consider for $\epsilon > 0$ the discrete convolution of any function in the class $\eta \in \lip_1(B_R)$ with $\|\eta\|_\infty \leq 1$ with $\chi_{\epsilon}(\cdot-x/N)$ and control the probability that $\cD^u_N$ and a deviation between $\scrL_{N,u}$ and $\scrM_{\mathring{A}}^u(x) \De x$ in the $d_{BL,R}$-distance of size bigger that $\Delta$ happen simultaneously. In essence, we show that the main contribution to this probability at leading order comes from the event that $|\langle \scrL_{N,u} - \scrM_{\mathring{A}}^u, \chi_\epsilon(\cdot - x/N) \rangle|$ becomes large for one of the $x \in B_{R,N}$ and disconnection occurs. This is encapsulated in the following proposition:
\begin{prop}\label{handling_supremum} Let $\chi:\bbR^d \to [0,1]$ be a symmetric, smooth probability density supported in the Euclidean unit ball and fix $R>0$.  Then, there exists $\epsilon = \epsilon(\Delta,R,A)\in (0,1)$ such that if $\chi_\epsilon(x):= \epsilon^{-d} \chi(x/\epsilon)$,
    \begin{equation}\label{eq:uniformity}
        \begin{aligned}
            \varlimsup_{N} &\tfrac{1}{N^{d-2}} \log \bbP\Big[d_{BL,R}(\scrL_{N,u}, \scrM^u_{\mathring{A}}) \geq \Delta; \cD^u_N\Big] \leq \big( -\tfrac{1}{d} (\sqrt{\overline{u}} -\sqrt{u})^2 \capa(\mathring{A}) - 1\big) 
            \\
            & \vee \varlimsup_{N} \sup_{x\in B_{R,N}}\tfrac{1}{N^{d-2}} \log \bbP\Big[\big|\langle\scrL_{N,u} - \scrM^u_{\mathring{A}}, \chi_{\epsilon}(\cdot-x/N)\rangle\big| \geq \tfrac{\Delta}{4|B_R|} ; \cD^u_N\Big].
        \end{aligned}
    \end{equation}
\end{prop}
\begin{proof} 
Recall the definition of $d_{BL,R}$ from~\eqref{bLDistance}. In order to reduce the family $\lip_1(B_R)$ of test functions to a much smaller set, we consider for each $\eta\in \lip_1(B_R)$ the \emph{discrete convolution}
\begin{equation}
    \label{eq:convolution}
    \eta^\epsilon_N(x) = \frac{1}{N^d} \sum_{y\in \bbZ^d} \chi_\epsilon\Big(x-\frac{y}{N}\Big) \eta\Big(\frac{y}{N}\Big),\quad x\in \bbR^d,
\end{equation}
where we implicitly extended $\eta$ to be zero outside of $B_R$. We note that $\supp \eta^\epsilon_N \subseteq B_{R+\epsilon}$, and 
\begin{equation}\label{eq:convolvedeta}
    \limsup_{N} \sup_{\substack{\|\eta\|_\infty\leq 1\\ \eta\in \lip_1(B_R)}} \sup_{x\in B_{R-\epsilon}} |\eta^\epsilon_N(x)-\eta(x)| \leq \epsilon \int \chi(y)|y|\, \De y.
\end{equation}
Our first goal is to replace $\eta$ by $\eta^\epsilon_N$ in $|\langle \scrL_{N,u} - \scrM^u_{\mathring{A}}, \eta\rangle|$. Upon using the triangle inequality, we have 
\begin{equation}
    \begin{split}
        d_{BL,R}&(\scrL_{N,u}, \scrM^u_{\mathring{A}}) = \sup_{\substack{\|\eta\|_\infty\leq 1\\ \eta\in \lip_1(B_R)}}|\langle \scrL_{N,u} - \scrM^u_{\mathring{A}}, \eta\rangle| \\
        & \leq \sup_{\substack{\|\eta\|_\infty\leq 1\\ \eta\in \lip_1(B_R)}}|\langle \scrL_{N,u} - \scrM^u_{\mathring{A}}, \eta^\epsilon_N\rangle|   
        + \big\langle \scrM^u_{\mathring{A}} + \scrL_{N,u}, \sup_{\substack{\|\eta\|_\infty\leq 1\\ \eta\in \lip_1(B_R)}}|\eta^\epsilon_N - \eta|\big\rangle.
    \end{split}
\end{equation}
Consequently, we obtain the bound
\begin{equation}\label{eq:split}
    \begin{aligned} 
        \bbP\Big[d_{BL,R}(\scrL_{N,u}, \scrM^u_{\mathring{A}})  \geq \Delta; \cD^u_N\Big] & \leq \bbP\Big[\sup_{\substack{\|\eta\|_\infty\leq 1\\ \eta\in \lip_1(B_R)}}|\langle \scrL_{N,u} - \scrM^u_{\mathring{A}}, \eta^\epsilon_N\rangle| \geq \tfrac{\Delta}{2}; \cD^u_N\Big] \\
        & + \bbP\Big[\big\langle \scrM^u_{\mathring{A}} + \scrL_{N,u}, \sup_{\substack{\|\eta\|_\infty\leq 1\\ \eta\in \lip_1(B_R)}}|\eta^\epsilon_N - \eta|\big\rangle \geq \tfrac{\Delta}{2}\Big].
    \end{aligned}
\end{equation}
We will now derive separate large deviations upper bounds for the two summands in~\eqref{eq:split}. We start with an upper bound for the first summand in~\eqref{eq:split}. We observe that in view of the definition of $\eta^\epsilon_N$ and the fact that $\|\eta\|_\infty\leq 1$ and $\supp \eta\subseteq B_R$
\begin{equation}
    \begin{split}
        |\langle \scrL_{N,u} - \scrM^u_{\mathring{A}},\eta^\epsilon_N\big\rangle|& \leq \frac{1}{N^d}\sum_{x\in \bbZ^d} \Big|\eta \Big(\frac{x}{N}\Big) \Big| \big|\langle\scrL_{N,u} - \scrM^u_{\mathring{A}}, \chi_{\epsilon}(\cdot-x/N)\rangle\big|\\
        &\leq \frac{|B_{R,N}|}{N^d} \sup_{x\in B_{R,N}}|\langle \scrL_{N,u} - \scrM^u_{\mathring{A}},  \chi_{\epsilon}(\cdot-x/N)\rangle|.
    \end{split}
\end{equation}Using a union bound, we thus conclude that 
\begin{equation}\label{eq:pre_firstsummand}
    \begin{aligned}
        \bbP\Big[&\sup_{\substack{\|\eta\|_\infty\leq 1\\ \eta\in \lip_1(B_R)}}|\langle \scrL_{N,u} - \scrM^u_{\mathring{A}}, \eta^\epsilon_N\rangle| \geq \tfrac{\Delta}{2}; \cD^u_N\Big] \\ &\leq  |B_{R,N}|\sup_{x\in B_{R,N}}  \bbP\Big[|\langle \scrL_{N,u} - \scrM^u_{\mathring{A}},  \chi_{\epsilon}(\cdot-x/N)\rangle| \geq \tfrac{\Delta N^d}{2 |B_{R,N}|}; \cD^u_N\Big].
    \end{aligned}
\end{equation}
Taking logarithms, dividing by $N^{d-2}$ and sending $N\to \infty$ in~\eqref{eq:pre_firstsummand} yields
\begin{equation}\label{eq:firstsummand}
    \begin{aligned}
        &\varlimsup_N \tfrac{1}{N^{d-2}} \log \bbP\Big[\sup_{\substack{\|\eta\|_\infty\leq 1\\ \eta\in \lip_1(B_R)}}|\langle \scrL_{N,u} - \scrM^u_{\mathring{A}}, \eta^\epsilon_N\rangle| \geq \tfrac{\Delta}{2}; \cD^u_N\Big] \\ &\leq  \varlimsup_N \sup_{x\in B_{R,N}}\tfrac{1}{N^{d-2}} \log   \bbP\Big[|\langle \scrL_{N,u} - \scrM^u_{\mathring{A}},  \chi_{\epsilon}(\cdot-x/N)\rangle| \geq \tfrac{\Delta}{4 |B_R|}; \cD^u_N\Big],
    \end{aligned}
\end{equation}
where we used that $|B_{R,N} | \leq 2 |B_R| N^d$ for any $N$ large enough.

We are left with handling the second summand in~\eqref{eq:split}. In view of~\eqref{eq:convolvedeta} and the fact that $\chi$ is assumed to be bounded by one, we have that for all $N$ large enough
\begin{equation}
\label{eq:MollificationCloseness}
    \sup_{\substack{\|\eta\|_\infty\leq 1\\\eta \in \lip_1(B_R)}} \sup_{x\in B_{R-\epsilon}} \big|\eta^\epsilon_N(x)- \eta(x) \big| \leq c \epsilon.
\end{equation}
Observe that $\eta^\epsilon-\eta=0$ on $\bbR^d\setminus B_{R+\epsilon}$. Then, for all large $N$, and any $\epsilon<1$,  
\begin{equation}\label{eq:boundlabel}
    \big\langle \scrM^u_{\mathring{A}} + \scrL_{N,u}, \sup_{\substack{\|\eta\|_\infty\leq 1\\ \eta\in \lip_1(B_R)}}|\eta^\epsilon_N - \eta|\big\rangle \leq c \epsilon \overline{u}\, R^d  + c \epsilon \langle \scrL_{N,u},  \mathbbm{1}_{B_{R-\epsilon}} + \tfrac{1}{\epsilon}\mathbbm{1}_{B_{R+\epsilon}\setminus B_{R-\epsilon}}\rangle.
\end{equation}

Define for convenience $\psi := \mathbbm{1}_{B_{R-\epsilon}} + \tfrac{1}{\epsilon}\mathbbm{1}_{B_{R+\epsilon}\setminus B_{R-\epsilon}}$ and $\psi_N (x) := \psi(x/N)$. We then fix $\lambda = \lambda(R)$ such that $\lambda N^{-2} \|G \psi_N\|_\infty \leq 1/2$ for all $N$ large. Let us explain why this choice is possible. For $x \in \bbZ^d$, one has
\begin{equation}
\label{eq:BoundGPsi}
\begin{split}
    G \psi_N(x) & = \sum_{y\in (N B_{R-\epsilon})\cap \bbZ^d} g(x,y) +  \tfrac{1}{\epsilon}\sum_{y\in (N B_{R+\epsilon}\setminus B_{R-\epsilon})\cap \bbZ^d} g(x,y) \\
    & \leq c R^2 N^2 + \tfrac{1}{\epsilon} \sum_{y\in (N B_{R+\epsilon}\setminus B_{R-\epsilon})\cap \bbZ^d} g(x,y).
    \end{split}
\end{equation}
Let us argue how the second summand in~\eqref{eq:BoundGPsi} is bounded. By $H^\epsilon_N = \{- 2 \lfloor \epsilon N \rfloor,..., 2 \lfloor \epsilon N \rfloor \} \times \{ - 2 \lfloor RN\rfloor ,..., 2\lfloor RN \rfloor \}^{d-1}$, we denote a slab of size  $4 \lfloor \epsilon N \rfloor + 1$ in one and  $4 \lfloor RN\rfloor + 1$ in $d-1$ dimensions.  For $N$ large enough, the set $(NB_{R+\epsilon} \setminus B_{R-\epsilon}) \cap \bbZ^d$ is contained in a union of $2d$ rotated and translated copies of $H_N^\epsilon$. By~\eqref{eq:AsymptoticBehaviourGreen} and the equivalence between the Euclidean and sup-norm in $\bbR^d$, one can see that 
\begin{equation}
\sum_{y\in (N B_{R+\epsilon}\setminus B_{R-\epsilon})\cap \bbZ^d} g(x,y) \leq \sum_{y \in H^\epsilon_N} \frac{c}{(|y|_\infty \wedge 1)^{d-2}},
\end{equation}
where we used that the sum of $g(x,y)$ over $y \in H_N^\epsilon$ is maximal when $x = 0$.

Finally, by decomposing the slab $H_N^\epsilon$ into the cube $Q_N^\epsilon$, defined by $Q_N^\epsilon := \{-2\lfloor \epsilon N\rfloor ,..., 2\lfloor \epsilon N \rfloor \}^d$, and the rest, one obtains:
\begin{equation}
\label{eq:BoundSecondTermSlab}
\begin{split}
\sum_{y \in H^\epsilon_N} \frac{c}{(|y|_\infty \wedge 1)^{d-2}} & \leq \sum_{y \in Q^\epsilon_N} \frac{c}{(|y|_\infty \wedge 1)^{d-2}} + \sum_{y \in H^\epsilon_N \setminus Q^\epsilon_N } \frac{c}{(|y|_\infty \wedge 1)^{d-2}} \\
& \leq c \epsilon^2 N^2 + c' \epsilon R N^2.
\end{split}
\end{equation} 
Combining~\eqref{eq:BoundGPsi}--\eqref{eq:BoundSecondTermSlab} and using $\epsilon < 1$, one finds that $\|G \psi_N \|_\infty \leq cR^2 N^2$.

Then, with the help of~\eqref{eq:laplacetransform}, we get for all $N$ large enough (using the exponential Markov inequality in the second bound)
\begin{equation}
    \begin{aligned}
        \bbP& \Big[\big\langle \scrM^u_{\mathring{A}} + \scrL_{N,u}, \sup_{\substack{\|\eta\|_\infty\leq 1\\ \eta\in \lip_1(B_R)}}|\eta^\epsilon_N - \eta|\big\rangle \geq \tfrac{\Delta}{2}\Big] \stackrel{\eqref{eq:boundlabel}}{\leq} \bbP\Big[\langle \scrL_{N,u},\psi\rangle \geq c'\tfrac{\Delta}{\epsilon} - c \overline{u}\,R^d\Big]\\ 
        &\leq \exp\Big\{-\lambda N^{d-2} \Big(c' \tfrac{\Delta}{\epsilon} - c\overline{u}\, R^d\Big) + u \lambda N^{-2} \langle \psi_N, (I - \lambda N^{-2} G \psi_N)^{-1}1\rangle_{\bbZ^d}\Big\}\\
        &\leq \exp\Big\{-\lambda N^{d-2} \Big( c'\tfrac{\Delta}{\epsilon} - c \overline{u}\,R^d\Big) + c u \lambda  R^d N^{d-2}\Big\}, 
    \end{aligned}
\end{equation}
We finally choose $\epsilon = \epsilon(\Delta,R,A) < 1$ in such a way that
\begin{equation}\label{eq:choiceofepsilon}
    \lambda \Big(c' \tfrac{\Delta}{\epsilon} - c \overline{u}\,R^d\Big) - c u \lambda R^d  \geq  \tfrac{1}{d} (\sqrt{\overline{u}} -\sqrt{u})^2 \capa(\mathring{A}) + 1.
\end{equation}

After taking logarithms and sending $N\to \infty$, we conclude that for this choice of $\epsilon$
\begin{equation}\label{eq:secondsummand}
\begin{split}
    \limsup_{N} \tfrac{1}{N^{d-2}} \log \bbP\Big[\big\langle \scrM^u_{\mathring{A}} + \scrL_{N,u}, & \sup_{\substack{\|\eta\|_\infty\leq 1\\ \eta\in \lip_1(B_R)}}|\eta^\epsilon_N - \eta|\big\rangle \geq \tfrac{\Delta}{2}\Big] \\
    & \leq - \tfrac{1}{d} (\sqrt{\overline{u}} -\sqrt{u})^2 \capa(\mathring{A}) - 1.
    \end{split}
\end{equation}
By combining~\eqref{eq:firstsummand} and~\eqref{eq:secondsummand} in~\eqref{eq:split}, we obtain the claimed~\eqref{eq:uniformity}. 
\end{proof}

\paragraph{\textbf{Step 2. Coarse graining of the disconnection event}} In this step, we revisit the coarse-graining of the disconnection event $\cD^u_N$ developed in \cite{nitzschner2017solidification} that allows us to bring into play a set of `bad' boxes between $A_N$ and $S_N$ whose scaled $\bbR^d$-filling will act as a porous interface.

Recall the definition of the scales $L_0$ and $\widehat{L}_0$ from \eqref{eq:ScalesL0}. As a first step, we define the random subset 
\begin{equation}
\cU^1 = \ \begin{minipage}{0.8\linewidth}
  the union of all $L_0$-boxes $B_z$ that are either contained in\\  $B(0,(M+1)N)^c$ or linked to an $L_0$-box in $B(0,(M+1)N)^c$ by
a path of $L_0$-boxes $B_{z_i}$, $0 \leq i \leq n$, all (except possibly the last one) good($\alpha,\beta,\gamma$) and such that $N_u(D_{z_i}) < \beta \capa_{\bbZ^d}(D_{z_i})$.
\end{minipage}
\end{equation}
To determine the presence of $\cU^1$ within boxes $B(x,\widehat{L}_0)$, we introduce the function
\begin{equation}
\widehat{\sigma}(x) = |\cU^1 \cap B(x,\widehat{L}_0) | / |B(x,\widehat{L}_0)|, \quad x \in \bbZ^d,
\end{equation}
and moreover we define the (random) subset $\widehat{\cS}_N \subseteq \widehat{\bbL}_0$, that provides a `segmentation' of the interface of blocking $L_0$-boxes, namely 
\begin{equation}
\widehat{\cS}_N = \left\{x \in \widehat{\bbL}_0; \widehat{\sigma}(x) \in \left[ \tfrac{1}{4}, \tfrac{3}{4} \right] \right\}.
\end{equation} 
We proceed as in (4.39) of~\cite{nitzschner2017solidification} and extract $\widetilde{\cS}_N \subseteq \widehat{\cS}_N $ such that
\begin{equation} 
\begin{minipage}{0.8\linewidth}$\widetilde{\cS}_N$ is a maximal subset of $\widehat{\cS}_N$ with the property that the $B(x,2\widehat{L}_0)$, $x \in \widetilde{\cS}_N$, are pairwise disjoint.
\end{minipage}
\end{equation}
Further, we need the `bad' event $\cB_N$ from (4.22) of \cite{nitzschner2017solidification}, which is defined as 
\begin{equation}
\label{eq:BadEvent}
\cB_N = \bigcup_{e \in \{e_1,...,e_d \} } \left\{ \begin{minipage}{0.6\textwidth}
  there are at least $\rho(L_0)(N_{L_0}/L_0)^{d-1}$ columns of $L_0$-boxes in direction $e$ in $B(0,10(M+1)N)$ that contain a bad($\alpha, \beta, \gamma$) $L_0$-box 
\end{minipage}\right\},
\end{equation}
in which $\rho(L)$ is a positive function depending on $\alpha, \beta, \gamma$ and $K$ tending to $0$ as $L \rightarrow \infty$, $N_{L_0} = L_0^{d-1}/\log L_0$, and $\{e_1,...,e_d \}$ is the canonical basis of $\bbR^d$. The probability of the `bad' event decays with a super-exponential rate at the order we are interested in, namely
\begin{equation}
\label{eq:SuperExponentialBound}
\lim_{N } \frac{1}{N^{d-2}} \log \bbP[\cB_N] = -\infty,
\end{equation}
see Lemma 4.2 of \cite{nitzschner2017solidification}. For this reason, it can be discarded in the further discussion by working on an \emph{effective} disconnection event 
\begin{equation}
\widetilde{\cD}^\alpha_N = \cD^\alpha_N \setminus \cB_N.
\end{equation}
We set $\overline{K} = 2K+3$ and as in (4.41) of \cite{nitzschner2017solidification}, one has that
\begin{equation}\label{effectiveevent}
  \begin{minipage}{0.85\textwidth}
    for large $N$, on $\widetilde{\cD}^\alpha_N$, for each $x \in \widetilde{\cS}_N$, there is a collection $\widetilde{\cC}_x$ of points in $\bbL_0$ such that the corresponding $L_0$-boxes $B_z$, $z\in \widetilde{\cC}_x$, intersect $B(x,\widehat{L}_0)$ with $\widetilde{\pi}_x$-projection at mutual distance at least $\overline{K}L_0$. Moreover, $\widetilde{\cC}_x$ has cardinality $\big\lfloor\big( \tfrac{c'}{K} \tfrac{\widehat{L}_0}{L_0} \big)^{d-1} \big\rfloor$ and for each $z \in \widetilde{\cC}_x$, $B_z$ is good($\alpha,\beta,\gamma$) and $N_u(D_z) \geq \beta \capa_{\bbZ^d}(D_z)$ \end{minipage}
\end{equation}
(here, for each $x \in \widetilde{\cS}_N$, $\widetilde{\pi}_x$ is the projection on the set of points in $\bbZ^d$ with vanishing $\widetilde{i}_x$-coordinate).  

We now introduce the random variable $\kappa_N$ defined on $\widetilde{\cD}^\alpha_N$ with range $\cK_N$, 
\begin{equation}
\kappa_N = \big(\widehat{\cS}_N, \widetilde{\cS}_N, (\widetilde{\pi}_x, \widetilde{\cC}_x)_{x \in \widetilde{\cS}_N} \big),
\end{equation}
see below (4.41) of \cite{nitzschner2017solidification} or (3.19) of \cite{nitzschner2018entropic}. As in (4.43) of \cite{nitzschner2017solidification} the choice of $\widehat{L}_0$ and $L_0$ together with \eqref{effectiveevent}, implies the `small combinatorial complexity':
\begin{equation}
\label{eq:SmallCombCompl}
|\cK_N| = \exp\{ o(N^{d-2}) \}.
\end{equation}  
For a given choice of $\kappa = (\widehat{\cS},\widetilde{\cS}, (\widetilde{\pi}_x,\widetilde{\cC}_x)_{x \in \widetilde{\cS}}) \in \cK_N$, we define a number of sets which will be of use later:
\begin{equation}
\label{eq:UnionBound}
\begin{cases}
\cC & = \bigcup_{x \in \widetilde{\cS}} \widetilde{\cC}_x \\
C & = \bigcup_{z \in \cC} D_z \subseteq \bbZ^d \\
U_1 & = \text{ the unbounded component of }\bbR^d \setminus \tfrac{1}{N}  \bigcup_{x \in \widehat{\cS}} B_\infty\Big(x, \tfrac{1}{50d} \widehat{L}_0 \Big)  \\
U_0 & = \bbR^d \setminus U_1 \\
\Sigma^{(r)} & = \frac{1}{N} \bigcup_{z \in \cC} \Big (z + [-(3+r)L_0, (4+r)L_0]^d \Big ) \subseteq \bbR^d, r \in (0,3) \\
\Sigma & = \frac{1}{N} \bigcup_{z \in \cC} \Big (z + [-3L_0,4L_0]^d \Big ) \subseteq \bbR^d \\
\Sigma_{(r)} & = \frac{1}{N} \bigcup_{z \in \cC} \Big (z + [-(3-r)L_0, (4-r)L_0]^d \Big )\subseteq \bbR^d, r \in (0,3).
\end{cases}
\end{equation}
\begin{figure}[ht]\label{fig:scalesDisc}
  \centering
  \includegraphics[width=0.7\textwidth]{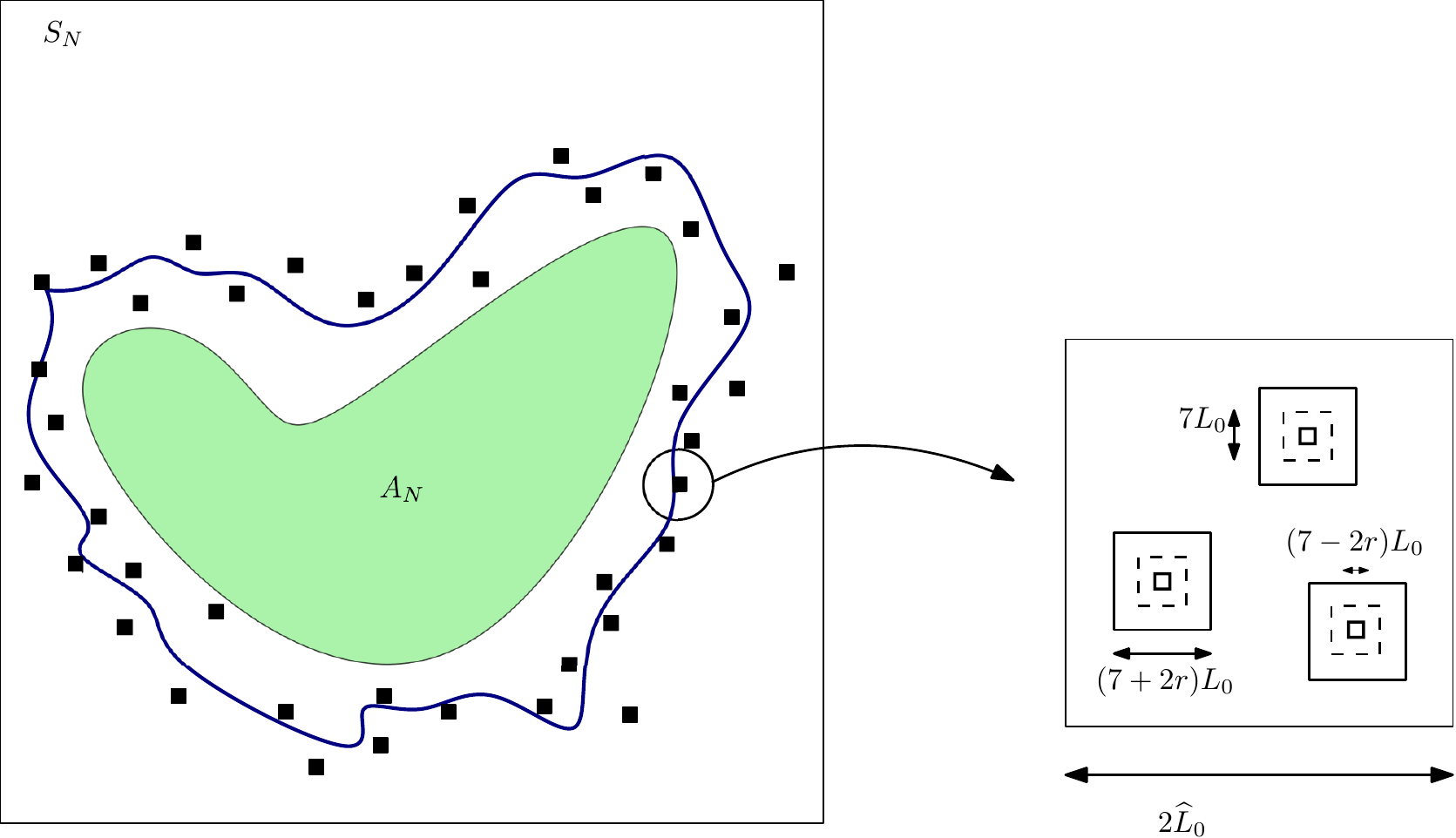} 
  \caption{Informal illustration of the boxes present for a given $\kappa \in \cK_N$, with the set of selected boxes of side-length $2\widehat{L}_0$ surrounding $A_N$ and the blow-up of one such box with selected boxes $D_z$ of size $7L_0$ (dashed lines) and slight enlargements/diminutions. The scaled $\bbR^d$-fillings of the dashed boxes constitute $\Sigma$, while the scaled $\bbR^d$-fillings of the enlarged/diminished boxes constitute $\Sigma^{(r)}$ and $\Sigma_{(r)}$, respectively.}
 \end{figure}

Note that for every $r \in (0,3)$, $\Sigma_{(r)} \subseteq \Sigma \subseteq \Sigma^{(r)} $, and $\Sigma_{(r)} $ and $ \Sigma^{(r)}$ are an enlargement and a diminution respectively, of the scaled $\bbR^d$-filling $\Sigma$ of $C$, and all three can be seen as `porous interfaces' for the `segmentation' $U_0$ in the sense of \eqref{eq:SegmentationClass} and \eqref{eq:ClassofporousInterf} (with the choice $\varepsilon = 10 \tfrac{\widehat{L}_0}{N}$). It should be noted that the occupation time bounds that will be developed in Step 4 force us to work with the boxes $D_z$, $z \in \cC$ as opposed to $B_z$, $z \in \cC$ in \cite{nitzschner2017solidification, sznitman2018macroscopic} to eventually build up the `porous interfaces' $\Sigma$.
This brings us to the main decomposition of this section, namely the coarse-graining
\begin{equation}\label{eq:coarseGraining}
\widetilde{\cD}^\alpha_N = \bigcup_{\kappa \in \cK_N} \cD_{N,\kappa}, \ \text{ where } \cD_{N,\kappa} = \widetilde{\cD}^\alpha_N \cap \{ \kappa_N = \kappa \}.
\end{equation}

We finish this step by introducing the set of `good' configurations $\cK^\mu_N$ for a given real number $\mu > 0$. In essence, we will declare a configuration $\kappa$ to be in $\cK^\mu_N$, if the harmonic potential $\scrh_{\Sigma}$ is `close' to $\scrh_{\mathring{A}}$ as measured by the Dirichlet form. The formal definition is
\begin{equation}
\label{eq:GoodInterfacesDef}
\cK^\mu_N = \{ \kappa \in \cK_N : \cE(\scrh_{\Sigma} - \scrh_{\mathring{A}})  \leq \mu \}.
\end{equation}

\paragraph{\textbf{Step 3. Uniform replacement of $\scrM^u_{\mathring{A}}$ by $\cM^u_C$}}
For the interfaces $C$ attached to $\kappa \in \cK^\mu_N$ (see \eqref{eq:UnionBound}), we define a discrete approximation to $\scrM^u_{\mathring{A}}$, namely the function 
\begin{equation}
\label{eq:DiscreteProfileBoxes}
\cM^u_C(x) = (\sqrt{u} + (\sqrt{\overline{u}} - \sqrt{u}) h_C(x))^2, \qquad x \in \bbZ^d
\end{equation}
(recall the definition of $h_C$ from~\eqref{eq:harmPotentialDiscrete}). Our aim is to show that for large $N, K$, uniformly in $\kappa \in \cK^\mu_N$, $\cM^u_{C}$ is (after scaling) a good approximation of $\scrM^u_{\mathring{A}}$. More precisely we show the following.
\begin{prop}
\label{ReplacementProp} For any fixed $\mu >0$, and $K \geq 100$ large enough (depending on $\mu$)
    \begin{equation}
        \label{eq:ReplacementLALC}
            \limsup_{N}\sup_{\kappa\in \cK^\mu_N} \sup_{\substack{\|\eta\|_\infty\leq 1\\ \eta\in \lip_1(B_R)}} \Big| \frac{1}{N^d} \sum_{x\in \bbZ^d} \cM^u_{C} (x) \eta(\tfrac{x}{N})-\langle \scrM^u_{\mathring{A}}, \eta\rangle\Big|\leq  8 E(\mathbbm{1}_{B_R})^{1/2} \mu^{1/2},
        \end{equation}
where $E$ was defined below~\eqref{eq:CSEnergies}.
\end{prop} 
\begin{proof}
By replacing $\eta$ with $-\eta$, we see that it suffices to obtain an upper bound for
\begin{equation}\label{eq:first_LA2LC}
    \begin{aligned}
    \frac{1}{N^d} \sum_{x\in \bbZ^d} \cM^u_{C} (x) \eta(\tfrac{x}{N})-\langle \scrM^u_{\mathring{A}}, \eta\rangle &= \Big(\frac{1}{N^d} \sum_{x\in \bbZ^d} \cM^u_{C} (x) \eta^+(\tfrac{x}{N})-\langle \scrM^u_{\mathring{A}}, \eta^+\rangle\Big) \\ 
    &- \Big(\frac{1}{N^d} \sum_{x\in \bbZ^d} \cM^u_{C} (x) \eta^-(\tfrac{x}{N})-\langle \scrM^u_{\mathring{A}}, \eta^-\rangle\Big).    
\end{aligned}
\end{equation}
This requires an upper bound on the first term on the right-hand side of~\eqref{eq:first_LA2LC} and a lower bound on the second term. We will give details only for the first contribution since the lower bound on the second contribution is obtained with a similar argument.

Combining the definitions of $\scrM^u_{\Sigma^{(r)}}$ in \eqref{eq:DefinitionScrL} and $\cM^u_C$ in \eqref{eq:DiscreteProfileBoxes} together with the fact that $u \in (0,\overline{u})$ and $h_C, \scrh_{\Sigma^{(r)}} \leq 1$, we obtain the inequality
\begin{equation}
 \cM^u_{C} (x) -  \scrM^u_{\Sigma^{(r)}} (\tfrac{x}{N}) \leq 2 \overline{u} ( h_C(x) - \scrh_{\Sigma^{(r)}}(\tfrac{x}{N})) \vee 0,
\end{equation} 
 from which one can easily infer that
 \begin{equation}\label{eq:strong coupling}
 \begin{split}
        \frac{1}{N^d} \sum_{x\in \bbZ^d} \cM^u_{C} (x)& \eta^+(\tfrac{x}{N})  \leq \frac{1}{N^d} \sum_{x\in \bbZ^d} \scrM^u_{\Sigma^{(r)}} (\tfrac{x}{N}) \eta^+(\tfrac{x}{N}) \\
        &+ \sup_{\kappa \in \cK_N^\mu} \sup_{x\in B_{R,N}} 2\overline{u} \frac{|B_{R,N}|}{N^d} ( h_C(x) - \scrh_{\Sigma^{(r)}}(\tfrac{x}{N})) \vee 0.
        \end{split}
\end{equation}
An application of (A.5) in~\cite[Proposition A.1]{chiarini2018entropic}, which relies on a strong coupling result of~\cite{einmahl1989extensions} between the simple random walk and Brownian motion, shows that the second summand of~\eqref{eq:strong coupling} converges to zero as $N\to \infty$ (in fact, one has to use a slight modification of this statement).

We now provide a gradient estimate to approximate the first summand in~\eqref{eq:strong coupling} by an integral, similar as in (4.32)--(4.35) of~\cite{chiarini2018entropic}. 
Let $\Sigma^{(r),\ast}$ and $\Sigma^{(r),\ast}_2$ be the enlargements of $\Sigma^{(r)}$ by $\tfrac{L_0}{N}$ and $\tfrac{L_0}{2N}$ respectively, that is
\begin{equation}
    \Sigma^{(r),\ast} = \{x\in \bbR^d:\, d(x,\Sigma^{(r)})\leq \tfrac{L_0}{N}\},\quad \Sigma^{(r),\ast}_2  = \{x\in \bbR^d:\, d(x,\Sigma^{(r)})\leq \tfrac{L_0}{2N}\}. 
\end{equation}
Using the fact that $\scrh_{\Sigma^{(r)}}$ is harmonic outside $\Sigma^{(r),\ast}_2$ and Theorem 2.10 in~\cite{gilbarg2015elliptic} yields the gradient bound
\begin{equation}\label{eq:gradientestimate}
    \sup_{x\in (\Sigma^{(r),\ast}_2)^c} |\nabla \scrM^u_{\Sigma^{(r)}} (x)| \leq \sup_{x\in (\Sigma^{(r),\ast}_2)^c} 2(\overline{u} - \sqrt{\overline{u} u}) |\nabla \scrh_{\Sigma^{(r)}}(x)| \leq c \frac{N}{L_0}.
\end{equation}
We proceed with the approximation of the sum in~\eqref{eq:strong coupling} by an integral. To this end, we write
\begin{equation}
\label{eq:Step3_part1}
    \frac{1}{N^d} \sum_{x\in \bbZ^d} \scrM^u_{\Sigma^{(r)}} (\tfrac{x}{N}) \eta^+(\tfrac{x}{N}) \leq \int \scrM^u_{\Sigma^{(r)}} (x) \eta^+(x)\,\De x + \sup_{\kappa\in \cK_N^\mu}  \sup_{\substack{\|\eta\|_\infty\leq 1\\ \eta\in \lip_1(B_R)}} \cR^N_{\eta,\kappa},
\end{equation}
where 
\begin{equation}
    \begin{aligned}
        \cR^N_{\eta,\kappa} &= \frac{1}{N^d} \sum_{y \in \bbZ^d : \tfrac{y}{N} \in \Sigma^{(r),\ast}} \eta^+(\tfrac{y}{N}) \scrM^u_{\Sigma^{(r)}} (\tfrac{y}{N})\\ & + \sum_{y \in \bbZ^d :\tfrac{y}{N} \notin \Sigma^{(r),\ast}} \int_{\big[\tfrac{y}{N},\tfrac{y+1}{N}\big)^d} \scrM^u_{\Sigma^{(r)}}(x) |\eta^+(\tfrac{y}{N})-\eta^+(x) |\,\De x\\ 
        &+    \sum_{y \in \bbZ^d : \tfrac{y}{N} \notin \Sigma^{(r),\ast}} \eta^+(\tfrac{y}{N}) \int_{\big[\tfrac{y}{N},\tfrac{y+1}{N}\big)^d} |\scrM^u_{\Sigma^{(r)}}(x)-\scrM^u_{\Sigma^{(r)}}(\tfrac{y}{N})| \,\De x.
    \end{aligned}
\end{equation}
Furthermore, from this representation of the error term, we obtain
\begin{equation}
\begin{split}
\label{eq:Step_part2}
    \varlimsup_{N} & \sup_{\kappa\in \cK_N^\mu}  \sup_{\substack{\|\eta\|_\infty\leq 1\\ \eta\in \lip_1(B_R)}} \cR^N_{\eta, \kappa} \\\
    & \leq \varlimsup_{N}\sup_{\kappa\in \cK_N^\mu}  \sup_{\substack{\|\eta\|_\infty\leq 1\\ \eta\in \lip_1(B_R)}} \Big( \frac{c|C|}{N^d}\|\eta\|_{\infty} + \frac{c}{N}\lip(\eta) + \frac{cN}{L_0}\|\eta\|_{\infty} \Big) = 0,
    \end{split}
\end{equation}
where we used that $\sup_{\kappa \in \cK^\mu_N} |C| = o(N^d)$ by construction (see e.g.~the argument in (3.39) of~\cite{chiarini2018entropic}) for the first summand, the Lipschitz continuity of $\eta$ for the second contribution and the fact that $\bigcup_{y\in \bbZ^d: y/N \notin \Sigma^{(r),\ast}} [\tfrac{y}{N}, \tfrac{y+1}{N}) \subseteq (\Sigma^{(r),\ast}_2)^c$ which allows the application of the gradient estimate \eqref{eq:gradientestimate} for the third summand.

Combining~\eqref{eq:strong coupling} and~\eqref{eq:Step3_part1}, we finally get
\begin{equation}\label{eq:upperbound_step3}
    \begin{aligned}
        \tfrac{1}{N^d} \sum_{x\in \bbZ^d} & \cM^u_{C} (x) \eta^+(\tfrac{x}{N}) - \int \scrM^u_{\mathring{A}}(x) \eta^+(x)\, \De x \\
        &  \leq \int (\scrM^u_{\Sigma^{(r)}}(x)-\scrM^u_{\mathring{A}}(x)) \eta^+(x)\,\De x + \cR_N \\
        &\leq 2 \overline{u} \int |\scrh_{\Sigma^{(r)}}(x)-\scrh_{\mathring{A}}(x)| \eta^+(x)\, \De x + \cR_N \\
        &\leq  2 \overline{u}E(|\eta|)^{1/2} \cE(\scrh_{\Sigma^{(r)}}-\scrh_{\mathring{A}})^{1/2} + \cR_N,
    \end{aligned}
\end{equation}
with the definition 
\[
    \cR_N = \sup_{\kappa \in \cK_N^\mu} \sup_{\substack{\|\eta\|_\infty\leq 1\\ \eta\in \lip_1(B_R)}} \cR^N_{\eta,\kappa} + \sup_{\kappa \in \cK_N^\mu} \sup_{x\in B_{R,N}} 2\overline{u} \frac{|B_{R,N}|}{N^d} ( h_C(x) - \scrh_{\Sigma^{(r)}}(\tfrac{x}{N})) \vee 0,
\] where in the last step we used~\eqref{eq:CSEnergies} and $E(\eta^+)\leq E(|\eta|)$. 

Via the same ideas (in particular using (A.4) in~\cite[Proposition A.1]{chiarini2018entropic}), we can show that
\begin{equation}\label{eq:lowerbound_step3}
\begin{split}
        \tfrac{1}{N^d} \sum_{x\in \bbZ^d}  \cM^u_{C} (x) \eta^-(\tfrac{x}{N}) & - \int \scrM^u_{\mathring{A}}(x) \eta^-(x)\, \De x \\
        &  \geq 
 - 2\overline{u} \big(E(|\eta|) \cE(\scrh_{\Sigma_{(r)}}-\scrh_{\mathring{A}})\big)^{1/2} - \cR'_N,
 \end{split}
\end{equation}
where $\cR'_N\to 0$ is independent of $\kappa\in \cK_N^\mu$ and $\eta\in \lip_1(B_R)$ with $\|\eta\|_\infty\leq 1$. 

Moreover, we note that for any $\kappa \in \cK_N^\mu$, one has
\begin{equation}
\label{eq:ReplacementOfEnlargedBoxes}
\begin{split}
\cE(\scrh_{\Sigma^{(r)}} - \scrh_{\mathring{A}})  &\leq 2\cE(\scrh_{\Sigma^{(r)}} - \scrh_{\Sigma}) + 2\cE(\scrh_{\Sigma} - \scrh_{\mathring{A}}) \\
& \leq 2(\capa(\Sigma^{(r)}) - \capa(\Sigma)) + 2\mu \\
& \leq 2 \capa(\Sigma) [\delta_{N,K}(1 + o_r(1)) -1] + 2\mu 
\\ &\leq c M^{d-2}  [\delta_{N,K}(1 + o_r(1)) -1]  + 2\mu,
\end{split}
\end{equation}
where $\delta_{N,K}\geq 0$ is a sequence fulfilling $\delta_{N,K} \rightarrow 1$ as $N,K \rightarrow \infty$, where we used \eqref{eq:GoodInterfacesDef} and the fact that $\Sigma \subseteq \Sigma^{(r)}$ in the second inequality, Proposition~\ref{ComparisonCapacities} in the third and $\capa(\Sigma) \leq \capa([-M,M]^d) = c M^{d-2}$ (by scaling) in the last. With the same argument, one can see that for any $\kappa \in \cK_N^\mu$:
\begin{equation}
\label{eq:ReplacementOfEnlargedBoxes2}
\begin{split}
\cE(\scrh_{\Sigma_{(r)}} - \scrh_{\mathring{A}}) &  \leq c M^{d-2}  [\delta_{N,K}(1 + o_r(1)) -1] + 2\mu.
\end{split}
\end{equation}
Upon inserting~\eqref{eq:upperbound_step3} and~\eqref{eq:lowerbound_step3} into~\eqref{eq:first_LA2LC}, we obtain
\begin{equation}
    \begin{aligned}
       & \varlimsup_{N} \widehat{\sup} \Big( \frac{1}{N^d} \sum_{x\in \bbZ^d} \cM^u_{C} (x) \eta(\tfrac{x}{N})-\langle \scrM^u_{\mathring{A}}, \eta\rangle \Big) \\ &\leq \varlimsup_{N} \Big( \widehat{\sup}\, 2\overline{u} E(|\eta|)^{1/2}  \big(\cE(\scrh_{\Sigma^{(r)}}-\scrh_{\mathring{A}})^{1/2} +  \cE(\scrh_{\Sigma_{(r)}}-\scrh_{\mathring{A}})^{1/2}\big) + \cR'_N+\cR_N \Big) \\ & 
        \stackrel{\eqref{eq:ReplacementOfEnlargedBoxes}, \eqref{eq:ReplacementOfEnlargedBoxes2}}{\leq}  4\overline{u} E(\mathbbm{1}_{B_R})^{1/2}  (c M^{d-2}   [\varlimsup_N \delta_{N,K}(1 + o_r(1)) -1]  + 2\mu) ^{1/2} \\
        & \leq 8\overline{u}E(\mathbbm{1}_{B_R})^{1/2} \mu^{1/2},
    \end{aligned}
\end{equation}
where we defined $\widehat{\sup}(\cdot) = \sup_{\kappa\in \cK_N^\mu}\sup_{\substack{\|\eta\|_\infty\leq 1\\ \eta\in \lip_1(B_R)}}(\cdot)$ and in the last step, $K \geq 100$ was chosen large enough and $r > 0$ small enough, which is what we wanted to show.
\end{proof}

\paragraph{\textbf{Step 4. Occupation-time bounds}}

    We consider $L_0$ as in \eqref{eq:ScalesL0}, $K\geq 100$ an integer and let
    $\cC$ be a non-empty finite subset of $\bbL_0$ with points at mutual distance at least $\overline{K} L_0$, where $\overline{K}=2K+3$ (for instance, the $\cC$ attached to a choice $\kappa \in \cK_N$ via~\eqref{eq:UnionBound} fulfills this condition).
    For a given $\cC$ as above we set $C = \bigcup_{z\in \cC} D_z$. We will often write $D\in \cC$ meaning $D=D_z$ with $z\in \cC$.

    The main result of this step is Proposition~\ref{PropOccTimeBounds} below which should be compared to Theorem 4.2 of~\cite{sznitman2017disconnection}. This result plays the role of Proposition 4.3 in~\cite{chiarini2018entropic}, where the notion of $h$-good box corresponds to $N_u(D)\geq \beta \capa_{\bbZ^d} (D)$ in the present context. Whereas Gaussian bounds (in particular the Borell-TIS inequality) were central in~\cite{chiarini2018entropic}, here instead we rely on the Laplace transform of the occupation times of random interlacements. 
    \begin{prop}
    \label{PropOccTimeBounds}
    Let $\alpha>\beta>\gamma $ in $(u,\overline{u})$ and $\zeta:\bbR^d \to \bbR$ such that $\|\zeta\|_\infty\leq 1$ and $\supp \zeta \subseteq B_R$.
    Define the event
    \begin{equation}
    \label{eq:GammaDef}
        \Gamma :=\bigcap_{z\in \cC} \{\text{$B_z$ is \normalfont{good}$(\alpha,\beta,\gamma)$ \textit{and} $N_u(D_z)\geq \beta \capa_{\bbZ^d}(D_z)$}\}.
    \end{equation}
    Then, there exists a function $K\mapsto \widehat{\varepsilon}(K)> 0$ such that $\widehat{\varepsilon}(K)\to 0$ as $K\to \infty$  and a positive constant $c_3(\Delta, R)$ such that 
    \begin{equation}\label{eq:exponentialbound}
        \begin{aligned}
     &\bbP\bigg[\Big|\langle \scrL_{N,u},\zeta \rangle - \tfrac{1}{N^d} \sum_{x\in \bbZ^d} \cM^u_C(x) \zeta(\tfrac{x}{N})\Big| \geq   \Delta;\, \Gamma\bigg]\\
     &\leq 2e^{ - \big[(\sqrt{\gamma} - \sqrt{u})^2 - \overline{u}\, \widehat{\varepsilon}(K) \big]\capa_{\bbZ^d}(C) - \big[c_3(\Delta, R) \sqrt{u} - c\, R^d  (\overline{u}-\gamma) \big] N^{d-2} }.
        \end{aligned}
    \end{equation}
    \end{prop}
    \begin{proof}
    In order to prove~\eqref{eq:exponentialbound}, it suffices show that 
   \begin{equation}\label{eq:exponentialboundwithoutAbsVal}
        \begin{aligned}
     &\bbP\bigg[\langle \scrL_{N,u},\zeta \rangle - \tfrac{1}{N^d} \sum_{x\in \bbZ^d} \cM^u_C(x) \zeta(\tfrac{x}{N}) \geq   \Delta;\, \Gamma\bigg]\\
     &\leq e^{ - \big[(\sqrt{\gamma} - \sqrt{u})^2 - \overline{u}\, \widehat{\varepsilon}(K) \big]\capa_{\bbZ^d}(C) - \big[c_3(\Delta, R) \sqrt{u} - c\, R^d  (\overline{u}-\gamma) \big] N^{d-2}}.
        \end{aligned}  
\end{equation}
    In fact an application of~\eqref{eq:exponentialboundwithoutAbsVal} to $\zeta$ and $-\zeta$, and a union bound to deal with the absolute value in~\eqref{eq:exponentialbound}, readily yields the conclusion.

    The proof of~\eqref{eq:exponentialboundwithoutAbsVal} will follow as an application of 
    Corollary~\ref{Cor_OccTime} to a perturbation of the potential 
    \begin{equation}
    \label{eq:VequalsaEc}
        V(x) = \frac{\sqrt{\gamma}-\sqrt{u}}{\sqrt{\gamma}} e_C(x), \qquad x\in \bbZ^d.
    \end{equation}
    For later convenience, we set $a = (\sqrt{\gamma}-\sqrt{u})/\sqrt{\gamma}$ so that $V = a e_C$. Notice that $a\in (0,1)$ since $\gamma\in (u,\overline{u})$. 
We start by encoding the event under the probability on the left-hand side of~\eqref{eq:exponentialbound} in an event involving a functional of a perturbation $V'$ of $V$, see \eqref{eq:encoded}.
Recall the notation $\cL_u = \sum_{x \in \bbZ^d} L_{x,u} \delta_{x} (= \scrL_{1,u})$, so that for a function $\zeta:\bbR^d \to \bbR$ one has
\begin{equation}
\label{eq:ScalarProductTwoLs}
    \langle \scrL_{N,u}, \zeta \rangle = \langle \cL_u, \tfrac{1}{N^d}\zeta(\tfrac{\cdot}{N})\rangle_{\bbZ^d}.  
\end{equation}
    If $B = B_z$, with $z\in \cC$, is good$(\alpha,\beta,\gamma)$ and $N_u(D)\geq \gamma \capa_{\bbZ^d}(D) $ (with $D=D_z$), then one has
    \begin{equation}\label{eq:occ_of_D}
        \langle \cL_u, e_D \rangle_{\bbZ^d} \geq \gamma \capa_{\bbZ^d}(D),
    \end{equation}
    see also (4.9) in Theorem 4.2 of~\cite{sznitman2017disconnection}.
    Thus on the event $\Gamma$, multiplying~\eqref{eq:occ_of_D} by $e_C(D)/\capa_{\bbZ^d}(D)$ and summing over $D\in \cC$ yields
    \begin{equation}
        \langle \cL_u ,\sum_{D\in \cC} e_C(D) \overline{e}_D \rangle_{\bbZ^d} \geq \gamma \capa_{\bbZ^d}(C)        
    \end{equation}
 (recall that $\overline{e}_D$ is the normalized equilibrium measure of $D$, cf.~\eqref{eq:NormEqMeausre}).   
    
    Finally, we use that by (4.6) in~\cite{sznitman2017disconnection} for all $K$ there exists $\widehat{\varepsilon}(K)>0$ such that $\widehat{\varepsilon}(K)\to 0$ as $K\to \infty$ and
    \begin{equation}\label{eq:perturbed potential}
        \sum_{D\in  \cC} e_C(D) \overline{e}_D(x) \leq (1+\widehat{\varepsilon}(K))e_C(x),\quad x\in\bbZ^d.
    \end{equation}
    This leads to the inclusion of events
    \begin{equation}\label{eq:encoded_disc}
            \Gamma\subseteq \big\{  \langle \cL_u, e_C\rangle_{\bbZ^d} \geq \tfrac{ \gamma \capa_{\bbZ^d} (C)}{1+\widehat{\varepsilon}(K) }\big\}  =  \big\{ \langle \cL_u, V\rangle_{\bbZ^d} \geq  \tfrac{\sqrt{\gamma}(\sqrt{\gamma}-\sqrt{u}) \capa_{\bbZ^d} (C)}{1+\widehat{\varepsilon}(K)} \big\}.
    \end{equation}
    We introduce now for $\delta\in (0,1)$ a perturbation $V'$ of the potential $V$
    \begin{equation}
        V'(x) = V(x) + \delta N^{-2} \zeta(\tfrac{x}{N}),\qquad x\in \bbZ^d,
    \end{equation}
    (here $N^{-2} \zeta(\tfrac{x}{N})$ will play the role of $\eta$ in Corollary~\ref{Cor_OccTime}).
    In view of~\eqref{eq:ScalarProductTwoLs}, \eqref{eq:perturbed potential} and~\eqref{eq:encoded_disc} we obtain the following bound
    \begin{equation}\label{eq:encoded}
        \begin{aligned}
           &  \bbP\bigg[\langle \scrL_{N,u},\zeta \rangle \geq \tfrac{1}{N^d} \sum_{x \in \bbZ^d} \cM^u_C(x) \zeta(\tfrac{x}{N}) + \Delta;\,\Gamma\bigg] \\ 
            &\leq \bbP\bigg[\langle \cL_u, V\rangle_{\bbZ^d} \geq  \tfrac{\sqrt{\gamma}(\sqrt{\gamma}-\sqrt{u})\capa_{\bbZ^d} (C)}{1+\widehat{\varepsilon}(K)}  ; \langle \scrL_{N,u},\zeta \rangle \geq \tfrac{1}{N^d} \sum_{x \in \bbZ^d} \cM^u_C(x) \zeta(\tfrac{x}{N})+ \Delta\bigg]\\
            & \leq \bbP\bigg[\langle \cL_u, V'\rangle_{\bbZ^d} \geq  \tfrac{\sqrt{\gamma}(\sqrt{\gamma}-\sqrt{u})\capa_{\bbZ^d} (C)}{1+\widehat{\varepsilon}(K)}  + \delta \langle \cM^u_C, N^{-2} \zeta(\tfrac{\cdot}{N})\rangle_{\bbZ^d}  + \delta N^{d-2} \Delta\bigg].
        \end{aligned}
    \end{equation}
    Our next goal is to rewrite the right-hand side of~\eqref{eq:encoded} in a way that allows the application of Corollary~\ref{Cor_OccTime}.
    By (2) of Remark~\ref{Remark_CalculationGamma} with $a = (\sqrt{\gamma}-\sqrt{u})/\sqrt{\gamma}$, one has
    \begin{equation}\label{eq:formula1}
        \gamma_V = 1 + \frac{a}{1-a} h_C,\quad u\langle V,\gamma_V^2\rangle_{\bbZ^d} = \sqrt{\gamma} (\sqrt{\gamma}-\sqrt{u}) \capa_{\bbZ^d}(C),
    \end{equation}
    and that 
    \begin{equation}\label{eq:formula2}
        u \cE_{\bbZ^d}(\gamma_V-1,\gamma_V -1) = u \tfrac{a^2}{(1-a)^2} \cE_{\bbZ^d}(h_C,h_C) = (\sqrt{\gamma}-\sqrt{u})^2 \capa_{\bbZ^d}(C).
    \end{equation}
    Moreover, using~\eqref{eq:DiscreteProfileBoxes} and~\eqref{eq:formula1} (and that $ h_C(x) \in [0,1]$ for all $x \in \bbZ^d$) it follows that 
    \begin{equation}\label{eq:formula3}
        \big|\cM^u_C(x) - u \gamma_V^2(x)\big| \leq \overline{u}-\gamma, \qquad \text{ for all }x\in \bbZ^d.
    \end{equation}
    Using~\eqref{eq:formula1},~\eqref{eq:formula3}, $\supp \ \zeta \subseteq B_R$ and $\delta < 1$, we obtain the bounds:
    \begin{equation}\label{eq:batman}\begin{aligned}
        \tfrac{\sqrt{\gamma}(\sqrt{\gamma}-\sqrt{u})}{1+\widehat{\varepsilon}(K)}  \capa_{\bbZ^d} (C) &=  u\langle V,\gamma_V^2\rangle_{\bbZ^d} - \sqrt{\gamma}(\sqrt{\gamma}-\sqrt{u}) \capa_{\bbZ^d}(C) \tfrac{\widehat{\varepsilon}(K)}{1+\widehat{\varepsilon}(K)}\\
        &\geq u\langle V,\gamma_V^2\rangle_{\bbZ^d} - \overline{u} \capa_{\bbZ^d}(C) \widehat{\varepsilon}(K),
    \end{aligned}
    \end{equation}
    \begin{equation}\label{eq:robin}\begin{aligned}
        \delta \langle \cM^u_C, N^{-2} \zeta(\tfrac{\cdot}{N})\rangle_{\bbZ^d}
        &\geq u \delta \langle N^{-2} \zeta(\tfrac{\cdot}{N}), \gamma_V^2\rangle_{\bbZ^d} - c\, R^d N^{d-2} (\overline{u}-\gamma).
        \end{aligned}
    \end{equation}
    By means of~\eqref{eq:batman} and~\eqref{eq:robin}, the right hand side of~\eqref{eq:encoded} is bounded above by
    \begin{equation}\label{eq:doom}\bbP\big[\langle \cL_u, V'\rangle_{\bbZ^d} \geq  u\langle V,\gamma_V^2\rangle_{\bbZ^d} +  u \delta \langle N^{-2} \zeta(\tfrac{\cdot}{N}), \gamma_V^2\rangle_{\bbZ^d} +t \big],
    \end{equation}
    where we set
    \begin{equation}
        t := \delta N^{d-2}\Delta - \overline{u}\, \widehat{\varepsilon}(K) \capa_{\bbZ^d}(C) - c\, R^d  (\overline{u}-\gamma) N^{d-2}.
    \end{equation}
    In order to apply Corollary~\ref{Cor_OccTime} to~\eqref{eq:doom}, we are left with the verification of the assumptions. On the one hand, $G |V|(x) = a h_C(x) \in [0,1)$ for any $x \in \bbZ^d$, since $\gamma>u$. On the other hand, since $\|\zeta \|_\infty \leq 1$, one has
    \begin{equation}\label{eq:difference}
        \big|\delta N^{-2} G |\zeta|(\tfrac{\cdot}{N})(x)\big| \leq \delta N^{-2} \sum_{y\in B_{R,N}} g(x,y) \leq c R^2 \delta, \quad \text{for all }x \in \bbZ^d.
    \end{equation} 
    Therefore, using~\eqref{eq:VequalsaEc}, \eqref{eq:difference},  and that $h_C(x) \in [0,1]$ for all $x\in \bbZ^d$, we get for all $N\geq 1$
    \begin{equation}
        \|G |V'| \|_\infty\leq \|G|V|\|_\infty + \| \delta N^{-2} G |\zeta|(\tfrac{\cdot}{N})\|_\infty < 1 - \tfrac{\sqrt{u}}{\sqrt{\gamma}} +  c R^2\delta,
    \end{equation}
    \begin{equation}
        \|(I-G|V|)^{-1} \delta N^{-2} G |\zeta|(\tfrac{\cdot}{N})\|_\infty \leq \tfrac{1}{1 - \|G|V|\|_\infty} \|\delta N^{-2} G |\zeta|(\tfrac{\cdot}{N}) \|_\infty 
        \leq c\,\tfrac{\sqrt{\gamma}}{\sqrt{u}} R^2 \delta. 
    \end{equation}
    Thus there exists $\delta_1 := c_4(R)\sqrt{u} < 1$ such that for all $\delta < \delta_1$, one has 
    \begin{equation}\label{eq:samurai}
        \|G |V'| \|_\infty < 1,\quad 
        \|(I-GV)^{-1} \delta N^{-2} G |\zeta|(\tfrac{\cdot}{N})\|_\infty < 1.
    \end{equation}
    Moreover, the rest $\cR_{\delta,\eta,V}$ with $\eta = N^{-2} \zeta(\tfrac{\cdot}{N})$ can be bounded as
    \begin{equation}
        \cR_{\delta,\eta,V} \leq c N^{d-2}  R^d \frac{\gamma}{u} \frac{\delta^2 \frac{\sqrt{\gamma}}{\sqrt{u}} R^2 }{1 - \delta c' \frac{\sqrt{\gamma}}{\sqrt{u}} R^2  }
    \end{equation}
    and thus there exists $\delta_2 = c_5(\Delta,R) \sqrt{u} \leq c_4(R) \sqrt{u}$ such that  for all $\delta\leq \delta_2$:
    \begin{equation}
        - \delta N^{d-2} \Delta +  u\cR_{\delta,\eta,V} \leq -\tfrac{\delta \Delta}{2} N^{d-2}. 
    \end{equation}
    In view of~\eqref{eq:samurai} and using~\eqref{eq:formula2} we can finally apply Corollary~\ref{Cor_OccTime} for any $\delta < \delta_2$ and obtain
    \begin{equation}
        \begin{aligned}
            \bbP&\big[\langle \cL_u, V'\rangle_{\bbZ^d} \geq  u\langle V,\gamma_V^2\rangle_{\bbZ^d} +  u \delta \langle N^{-2} \zeta(\tfrac{\cdot}{N}), \gamma_V^2\rangle_{\bbZ^d} +t\big] \\
            &\leq \exp\Big\{ - (\sqrt{\gamma} - \sqrt{u})^2 \capa_{\bbZ^d}(C) - t + u \cR_{\delta, \eta, V} \Big\} \\ 
            & \leq \exp\Big\{ - \big[(\sqrt{\gamma} - \sqrt{u})^2 - \overline{u}\, \widehat{\varepsilon}(K) \big]\capa_{\bbZ^d}(C) - \big[\tfrac{\delta_2 \Delta}{2} - cR^d  (\overline{u}-\gamma) \big] N^{d-2}  \Big\},
        \end{aligned}
    \end{equation}
    which is what we wanted to prove by setting $ c_3(\Delta,R) := c_5(\Delta,R) \Delta/2$.
\end{proof}

\paragraph{\textbf{Step 5. Application of the solidification bounds}}
In this last step, we will essentially put together the results of the previous steps and use the solidification estimates to finalize the proof of Theorem~\ref{thm:MainResult}. This will prominently feature an argument involving a distinction between two types of $\kappa \in \cK_N$, that either give rise to `good' or `bad' interfaces $\Sigma$, see around~\eqref{eq:BoundBadGood}. Both cases will be dealt with separately, and in the (more delicate) case of `good' interfaces we need to invoke Propositions~\ref{ReplacementProp} and~\ref{PropOccTimeBounds} to obtain the required bounds.

We start with bounding the probability on the left-hand side in equation~\eqref{eq:MainResultSection4} by using Lemma~\ref{comparisonDistances}:
\begin{equation}
\label{eq:Step0}
\begin{split}
\bbP\big[ d_R(\scrL_{N,u}, \scrM^u_{\mathring{A}}) \geq \Delta ; \cD^u_N \big] & \leq \bbP\big[ d_{BL,R}(\scrL_{N,u}, \scrM^u_{\mathring{A}}) \geq \overline{\Delta} ; \cD^u_N \big],\\
\text{ with } \overline{\Delta} & = \frac{\Delta}{1 + 2 \frac{(\sqrt{d}R) \vee 1}{\int_{B_R} \scrM^u_{\mathring{A}}(x) \De x } }.
\end{split}
\end{equation}

Upon inspection of~\eqref{eq:uniformity}, we will now show that (using again the notation $B_{R,N} = (NB_R) \cap \bbZ^d$) 
\begin{equation}
\label{eq:CommonBoundOnSplit}
\begin{split}
\limsup_N \frac{1}{N^{d-2}} \sup_{x \in B_{R,N}} \log & \ \bbP\Big[\big|\langle\scrL_{N,u} - \scrM^u_{\mathring{A}}, \chi_{\epsilon}(\cdot-x/N)\rangle\big| \geq \tfrac{\overline{\Delta}}{4|B_R|} ; \cD^u_N\Big] \\
&\leq - \frac{1}{d} (\sqrt{\overline{u}} - \sqrt{u})^2 \capa(\mathring{A}) - c_6(\Delta,R,A,u),
\end{split}
\end{equation}
with $c_6(\Delta,R,A,u) \sim c_7(\Delta,R,A) \sqrt{u}$ as $u \rightarrow 0$. For a continuous, compactly supported function $\zeta : \mathbb{R}^d \rightarrow \mathbb{R}$ and $\widetilde{\Delta} > 0$, we introduce the shorthand notation
\begin{equation}
\cA_N^{\widetilde{\Delta}, \zeta} := \big\{ \big|\langle\scrL_{N,u} - \scrM^u_{\mathring{A}}, \zeta \rangle\big| \geq \widetilde{\Delta} \big\}.
\end{equation}
By \eqref{eq:SuperExponentialBound}, \eqref{eq:SmallCombCompl} and \eqref{eq:coarseGraining}, we see that with the notation $\Delta' = \frac{\overline{\Delta}}{4|B_R|}$:
\begin{equation}
\begin{split}
\limsup_N & \frac{1}{N^{d-2}}  \sup_{x \in B_{R,N}} \log  \ \bbP\big[\cA_N^{\Delta', \chi_\epsilon(\cdot - x/N)} \cap \cD^u_N \big] \\
& \leq \limsup_N \frac{1}{N^{d-2}} \sup_{x \in B_{R,N}} \sup_{\kappa \in \cK_N} \log \bbP[\cA_N^{\Delta', \chi_\epsilon(\cdot - x/N)} \cap \cD_{N,\kappa}].
\end{split}
\end{equation}
We will bound the term on the right-hand side of the above equation in different ways, depending on the nature of $\kappa \in \cK_N$: Recall the definition of the subset $\cK_N^\mu \subseteq \cK_N$ for $\mu > 0$ in~\eqref{eq:GoodInterfacesDef}. For `bad' situations $\kappa \in \cK_N \setminus \cK^\mu_N$, in which $\scrh_{\Sigma}$ differs from $\scrh_{\mathring{A}}$ by an amount exceeding $\mu$ as measured by the Dirichlet form, it suffices to focus on the event $\cD_{N,\kappa}$ alone to produce an additional cost in the exponential decay using the solidification result~\eqref{eq:SolidificationDirichletEnergy}. In `good' cases for $\kappa \in \cK^\mu_N$, we have to employ the uniform replacement of $\scrM^u_{\mathring{A}}$ by $\cM^u_C$ (see Proposition~\ref{ReplacementProp} of Step 3) and the occupation-time bounds (see Proposition~\ref{PropOccTimeBounds} of Step 4) and  argue that the relevant bounds hold uniformly when $x$ varies in $NB_R \cap \bbZ^d$. More precisely, we let $\mu > 0$ and bound
\begin{equation}
\label{eq:BoundBadGood}
\begin{split}
 \limsup_N& \frac{1}{N^{d-2}} \sup_{x \in B_{R,N}}  \sup_{\kappa \in \cK_N} \log \bbP[\cA_N^{\Delta', \chi_\epsilon(\cdot - x/N)} \cap \cD_{N,\kappa}] \leq \\
&\Big(\limsup_N\frac{1}{N^{d-2}}\sup_{x \in B_{R,N}} \sup_{\kappa \in \cK_N^\mu} \log \bbP[\cA_N^{\Delta', \chi_\epsilon(\cdot - x/N)} \cap \cD_{N,\kappa}]\Big) \\
 &\vee  \Big( \limsup_N\frac{1}{N^{d-2}} \sup_{\kappa \in \cK_N \setminus \cK^\mu_N} \log \bbP[\cD_{N,\kappa}] \Big).
\end{split} 
\end{equation} 
Let us first focus on the second part of the dichotomy, pertaining to the situation where $\kappa$ gives rise to a `bad' interface $\Sigma \subseteq \bbR^d$. By (4.52) of \cite{nitzschner2017solidification} and the argument leading up to it, one knows that for $\widetilde{\varepsilon} > 0$ small enough, 
\begin{equation}
\label{eq:lowerBoundbadCase}
\begin{split}
&\varlimsup_N\frac{1}{N^{d-2}} \sup_{\kappa \in \cK_N \setminus \cK^\mu_N} \log \bbP[\cD_{N,\kappa}] \\
& \leq - \Big( \sqrt{\gamma} - \tfrac{\sqrt{u}}{1 - \widetilde{\varepsilon}(\sqrt{\frac{\overline{u}}{u}} - 1)} \Big)(\sqrt{\gamma}- \sqrt{u}) \frac{1}{N^{d-2}} \varliminf_N \inf_{\kappa \in \cK_N \setminus \cK^\mu_N} \capa_{\bbZ^d}(C).
\end{split}
\end{equation}
Upon taking $\liminf_K$ in \eqref{eq:lowerBoundbadCase} and using Proposition A.1 of \cite{nitzschner2017solidification}, we get 
\begin{equation}
\label{eq:lowerBoundbadCase2}
\begin{split}
& \varlimsup_N\frac{1}{N^{d-2}} \sup_{\kappa \in \cK_N \setminus \cK^\mu_N} \log \bbP[\cD_{N,\kappa}] \\
& \leq - \Big( \sqrt{\gamma} - \tfrac{\sqrt{u}}{1 - \widetilde{\varepsilon}(\sqrt{\frac{\overline{u}}{u}} - 1)} \Big)(\sqrt{\gamma}- \sqrt{u}) \frac{1}{d} \varliminf_K \varliminf_N \inf_{\kappa \in \cK_N \setminus \cK^\mu_N} \capa(\Sigma).
\end{split}
\end{equation}
At this point, we will make use of the solidification result for Dirichlet forms,~\eqref{eq:SolidificationDirichletEnergy}. Let us sketch briefly, how $\Sigma$ can be interpreted as a porous interface: Recall the definition of the sets $U_0$ and $U_1$ associated to $\kappa \in \cK_N$ in~\eqref{eq:UnionBound}. Take $A' \subseteq \mathring{A}$ compact, and $\ell_\ast \geq 0$ (depending on $A, A'$) with the property that for large $N$ and all $\kappa \in \cK_N$, $d(A',U_1) \geq 2^{-\ell_\ast}$, see (4.50) of~\cite{nitzschner2017solidification}. By definition of $\kappa \in \cK_N$, one can argue with the help of a projection argument that for all $x \in \widetilde{\cS}$, $\capa(\bigcup_{z \in \cC_x} z + [-3L_0,4L_0]^d) \geq c(K)\widehat{L}_0^{d-2}$, thus one has
\begin{equation}
\label{eq:HittingProbPorous}
W_x[H_{\Sigma} < \tau_{10\tfrac{\widehat{L}_0}{N} }] \geq c(K), \text{ for all } x \in \partial U_1.
\end{equation}
We refer to (4.51) of~\cite{nitzschner2017solidification} for details of this argument. This ensures that $\Sigma$ for any choice of $\kappa \in \cK_N$ is a porous interface around $A'$ in the sense of \eqref{eq:ClassofporousInterf} and we can apply the solidification estimate~\eqref{eq:SolidificationDirichletEnergy}:
\begin{equation}
\label{eq:BoundSigma_LargeDirichletform}
\begin{split}
&\varliminf_N  \inf_{\kappa \in \cK_N \setminus \cK_N^\mu} \capa(\Sigma)  = \capa(A') + \varliminf_N \inf_{\kappa \in \cK_N \setminus \cK_N^\mu} (\capa(\Sigma) - \capa(A')) \\
&= \capa(A') + \varliminf_N \inf_{\kappa \in \cK_N \setminus \cK_N^\mu} \!\!\left( \cE(\scrh_{A'} - \scrh_{\Sigma}) + \capa(\Sigma) - \capa(A') - \cE(\scrh_{A'} - \scrh_{\Sigma})  \right) \\
&\stackrel{\eqref{eq:SolidificationDirichletEnergy}}{\geq} \capa(A') +  \varliminf_N \inf_{\kappa \in \cK_N \setminus \cK_N^\mu} \cE(\scrh_{A'} - \scrh_{\Sigma}) \\
&\geq \capa(A') + (\mu^{\frac{1}{2}} - \cE(\scrh_{\mathring{A}}-\scrh_{A'})^{\frac{1}{2}})_+^2,
\end{split}
\end{equation}
having used the inverse triangle inequality in the last step, see also (4.49)--(4.52) of~\cite{chiarini2018entropic} for a similar argument in the case of level-set percolation of the Gaussian free field. Combining~\eqref{eq:lowerBoundbadCase2} with~\eqref{eq:BoundSigma_LargeDirichletform} provides us in the case of `bad' interfaces $\kappa \in \cK_N \setminus \cK_N^\mu$ with the bound
\begin{equation}
\begin{split}
\label{eq:EndResultBadInterfaces}
&\limsup_N  \frac{1}{N^{d-2}} \log \bbP[\cD_{N,\kappa}] \\
& \leq -\Big( \sqrt{\gamma} - \tfrac{\sqrt{u}}{1 - \widetilde{\varepsilon}(\sqrt{\frac{\overline{u}}{u}} - 1)} \Big) \tfrac{\sqrt{\gamma}- \sqrt{u}}{d} \big(\capa(A') + (\mu^{\frac{1}{2}} - \cE(\scrh_{\mathring{A}}-\scrh_{A'})^{\frac{1}{2}})_+^2 \big).
\end{split}
\end{equation}

We will now turn to the more delicate bounds in the case of $\kappa \in \cK^\mu_N$, giving rise to `good' 
porous interfaces $\Sigma$, that is, those for which $\scrh_{\Sigma}$ and $\scrh_{A}$ are `close'. Here, we need to find a suitable upper bound on the first member of the maximum in \eqref{eq:BoundBadGood}. This will be done by combining Propositions~\ref{ReplacementProp} and~\ref{PropOccTimeBounds} from the previous two steps. Let $\mu > 0$ be small enough (depending only on $R$, $\Delta'$ and $A$) such that
\begin{equation}
\label{eq:Delta''}
\Delta' \tfrac{\epsilon^{d+1}}{\lip(\chi)\vee 1} - 8\overline{u} E(\mathbbm{1}_{B_{2R}})^{1/2} \mu^{1/2} > \tfrac{\Delta'}{2} \cdot  \tfrac{\epsilon^{d+1}}{\lip(\chi)\vee 1} =: \Delta''.
\end{equation}
For $N \geq N_0$ we conclude from the `uniform closeness' between $\scrM^u_{\mathring{A}}$ and $\cM^u_C$ \eqref{eq:ReplacementLALC} that  
\begin{equation}
\label{eq:InclusionProperty}
\cA^{2\Delta'', \zeta}_N \subseteq \bigg\{ \bigg\vert\langle \scrL_{N,u},\zeta\rangle - \frac{1}{N^d} \sum_{x \in \bbZ^d} \cM^u_C(x) \zeta(\tfrac{x}{N}) \bigg\vert \geq \Delta''  \bigg\},
\end{equation}
for every $\zeta \in \lip_1(B_R)$ with $\|\zeta\|_\infty \leq 1$ and every $\kappa \in \cK^\mu_N$. In particular,~\eqref{eq:InclusionProperty} holds for $\zeta = \tfrac{\chi_\epsilon(\cdot - y) \epsilon^{d+1}}{\lip(\chi)\vee 1}$, uniformly in $y\in B_R$. By \eqref{effectiveevent}, every choice of $\kappa \in \cK_N$ (and a fortiori every choice in $\kappa \in \cK^\mu_N$) gives rise to a system of boxes $B_z$ located at $z \in \cC \subseteq \bbZ^d$ such that every $B_z$ is good$(\alpha,\beta,\gamma)$ and fulfills $N_u(D_z) \geq \beta \capa_{\bbZ^d}(D_z)$. This brings us into a position where we can use the main result of Step 4, Proposition~\ref{PropOccTimeBounds}. Indeed, with the shorthand notation $\widetilde{\sup}( \cdot ) = \sup_{x \in B_{R,N}} \sup_{\kappa \in \cK_N^\mu}( \cdot )$, the first member of the maximum in \eqref{eq:BoundBadGood} is bounded as follows:
\begin{equation}
\begin{split}
& \varlimsup_N\tfrac{1}{N^{d-2}}\widetilde{\sup} \log \bbP[\cA_N^{\Delta', \chi_\epsilon(\cdot - x/N)} \cap \cD_{N,\kappa}] \\
& \leq \varlimsup_N\tfrac{1}{N^{d-2}} \widetilde{\sup} \log \bbP\bigg[\Big| \langle \scrL_{N,u},\chi_\epsilon(\cdot -\tfrac{x}{N}) \rangle -  \langle \cM^u_C, \tfrac{\chi_\epsilon(\tfrac{\cdot}{N} -\tfrac{x}{N})}{N^d} \rangle_{\bbZ^d} \Big| \geq \tfrac{\Delta'}{2}; \cD_{N,\kappa}\bigg] \\
& \leq \varlimsup_N\tfrac{1}{N^{d-2}}\widetilde{\sup} \log \bbP\bigg[ \Big| \langle \scrL_{N,u},\chi_\epsilon(\cdot -\tfrac{x}{N}) \rangle - 
\langle \cM^u_C, \tfrac{\chi_\epsilon(\tfrac{\cdot}{N} -\tfrac{x}{N})}{N^d} \rangle_{\bbZ^d} \Big|  \geq \tfrac{\Delta'}{2}; \Gamma
\bigg],
\end{split}
\end{equation}
where we used~\eqref{eq:InclusionProperty} in the first inequality and the notation~\eqref{eq:GammaDef}. We recall that for fixed $\epsilon > 0$, the functions $\chi_\epsilon(\cdot - \tfrac{x}{N})$, as $x$ varies in $B_{R,N}$ form a location family and have a common support contained in the $2\epsilon$-neighborhood of $B_R$. Moreover, the sup-norm $\|\chi_\epsilon(\cdot - \tfrac{x}{N})\|_\infty$ is independent of $x \in B_{R,N}$. Therefore, an application of the occupation-time bound~\eqref{eq:exponentialbound} yields
\begin{equation}
\label{eq:EndResultGoodInterfaces}
\begin{split}
& \varlimsup_N\frac{1}{N^{d-2}} \sup_{x \in B_{R,N}} \sup_{\kappa \in \cK_N^\mu} \log \bbP[\cA_N^{\Delta', \chi_\epsilon(\cdot - x/N)} \cap \cD_{N,\kappa}] \\
& \leq  - \varliminf_N   \inf_{\kappa \in \cK_N^\mu}  \big[(\sqrt{\gamma} - \sqrt{u})^2 - \overline{u}\, \widehat{\varepsilon}(K) \big]\tfrac{\capa_{\bbZ^d}(C)}{N^{d-2}} - c_3(\tfrac{\Delta'}{2}, R)\sqrt{u} + cR^d  (\overline{u}-\gamma)  .
\end{split}
\end{equation}
As before, we can take $\liminf_K$ and arrive with the help of Proposition A.1 of \cite{nitzschner2017solidification} at 
\begin{equation}
\begin{split}
& \varlimsup_N\tfrac{1}{N^{d-2}} \sup_{x \in B_{R,N}} \sup_{\kappa \in \cK_N^\mu} \log \bbP[\cA_N^{\Delta', \chi_\epsilon(\cdot - x/N)} \cap \cD_{N,\kappa}] \\
& \leq  - \varliminf_N \varliminf_K \tfrac{1}{N^{d-2}} \inf_{\kappa \in \cK_N^\mu} \tfrac{1}{d} (\sqrt{\gamma} - \sqrt{u})^2 \capa(\Sigma) - \big[c_3(\tfrac{\Delta'}{2}, R)\sqrt{u} - cR^d  (\overline{u}-\gamma) \big].
\end{split}
\end{equation}
We can now argue as above \eqref{eq:HittingProbPorous} and see that $\Sigma$ is again in the class of porous interfaces for $A' \subseteq \mathring{A}$ compact and $U_0$, $U_1$ as in \eqref{eq:UnionBound}.
By the capacity lower bound \eqref{eq:SolidificationEstimate}, we obtain that 
\begin{equation}
\liminf_N \inf_{\kappa \in \cK_N} \capa(\Sigma) \geq \capa(A').
\end{equation}
To conclude the proof, we go back to \eqref{eq:BoundBadGood} and combine the two upper bounds~\eqref{eq:EndResultBadInterfaces} and~\eqref{eq:EndResultGoodInterfaces}: 
\begin{equation}
\begin{split}
& \limsup_N \frac{1}{N^{d-2}} \sup_{x \in B_{R,N}} \sup_{\kappa \in \cK_N} \log \bbP[\cA_N^{\Delta', \chi_\epsilon(\cdot - x/N)} \cap \cD_{N,\kappa}] \leq \\
&-\Big( \sqrt{\gamma} - \tfrac{\sqrt{u}}{1 - \widetilde{\varepsilon}(\sqrt{\frac{\overline{u}}{u}} - 1)} \Big) \tfrac{\sqrt{\gamma}- \sqrt{u}}{d} \big(\capa(A') + (\mu^{\frac{1}{2}} -\cE(\scrh_{\mathring{A}}-\scrh_{A'})^{\frac{1}{2}})_+^2 \big)\\
&  \vee -\left(\frac{1}{d} (\sqrt{\gamma} - \sqrt{u})^2 \capa(A') - \big[c_3(\tfrac{\Delta'}{2}, R)\sqrt{u} - c R^d  (\overline{u}-\gamma) \big] \right).
\end{split}
\end{equation}
By letting successively $\widetilde{\varepsilon}$ tend to zero and $\alpha,\beta, \gamma$ to $\overline{u}$, and finally $A' \uparrow \mathring{A}$, we obtain, in view of~\eqref{eq:increasing_capacity} and of $\cE(\scrh_{\mathring{A}}-\scrh_{A'})\leq \capa(\mathring{A}) - \capa(A')$
\begin{equation}
\begin{split}
& \limsup_N \frac{1}{N^{d-2}} \sup_{x \in B_{R,N}} \sup_{\kappa \in \cK_N} \log \bbP[\cA_N^{\Delta', \chi_\epsilon(\cdot - x/N)} \cap \cD_{N,\kappa}] \leq \\
& - \frac{1}{d}(\sqrt{\overline{u}} - \sqrt{u})^2 \capa(\mathring{A}) - c_3(\tfrac{\Delta'}{2},R)\sqrt{u}  \wedge \frac{\mu}{d} (\sqrt{\overline{u}} - \sqrt{u})^2 \capa(\mathring{A}),
\end{split}
\end{equation}
and the claim~\eqref{eq:CommonBoundOnSplit} follows by choosing $\mu > 0$ such that~\eqref{eq:Delta''} holds and by setting $c_6(\Delta,R,A,u) :=  c_3(\tfrac{\Delta'}{2},R)\sqrt{u}  \wedge \frac{\mu}{d} (\sqrt{\overline{u}} - \sqrt{u})^2 \capa(\mathring{A})$. Inserting this bound into~\eqref{eq:uniformity} now results in 
 \begin{equation}
        \begin{aligned}
            \varlimsup_{N} \frac{1}{N^{d-2}} &\log \bbP\Big[d_{BL,R}(\scrL_{N,u}, \scrM^u_{\mathring{A}}) \geq \overline{\Delta}; \cD^u_N\Big] \\
            & \leq - \frac{1}{d}(\sqrt{\overline{u}} -\sqrt{u})^2 \capa(\mathring{A}) - c_6(\Delta,R,A,u) \wedge 1.
        \end{aligned}
    \end{equation}
By~\eqref{eq:Step0} we can now set $c_1(\Delta,R,A,u) := c_6(\Delta,R,A,u) \wedge 1$ and since $c_6(\Delta,R,A,u) \sim c_7(\Delta,R,A)\sqrt{u}$ as $u \rightarrow 0$, one also has that $c_1(\Delta,R,A,u) \sim c_2(\Delta,R,A)\sqrt{u}$. This finishes the proof of the Theorem.
\hfill $\square$

\begin{remark}
\label{Remark_RW}
One may wonder, whether Theorem \ref{thm:MainResult} or Corollary \ref{thm:CorollaryPush} provide also some insight into the behavior of the occupation-time profile of a random walk when conditioned on disconnecting $A_N$ from $S_N$.
Heuristically, the random walk can be interpreted as limit of random interlacements with vanishing intensity. In~Corollary 6.4 of \cite{sznitman2015disconnection} and Corollary 4.4 of~\cite{nitzschner2017solidification} asymptotic large deviation upper bounds on the disconnection probability for the random walk started at the origin were obtained. These bounds relied on a coupling between random interlacements conditioned on $0 \in \cI^u$ at arbitrarily small $u > 0$ and the random walk. Similar coupling strategies have also been helpful in the discussion of the existence of macroscopic holes within a large box created by a random walk, see Section 4 of~\cite{sznitman2018macroscopic}. 

However, the specific form of the additional cost for disconnection by random interlacements and a deviation between $\scrL_{N,u}$ and $\cM^u_{\mathring{A}}$ involves an explicit dependence on $u$ that does not allow a simple coupling argument to conclude that for $A$ regular, also the occupation-time profile of a simple random walk, and $\scrM_A = u_\ast\scrh_{A}^2$ have to be close conditionally on disconnection and assuming $\overline{u} = u_\ast = u_{\ast\ast}$. In spite of this, one may still ask whether it is true that
\begin{equation}
\lim_N E_0\big[d_R(\scrL_N, \scrM_{\mathring{A}}) \wedge 1 | \cD_N \big] = 0, \text{ where } \cD_N = \Big\lbrace A_N \stackrel{\mathrm{Range}\{(X_t)_{t \geq 0} \}^c }{\centernot\longleftrightarrow} S_N \Big\rbrace?
\end{equation}
In investigating such a question, it should be noted that the strategy to obtain the lower bound on the disconnection probability for random walk in~\cite{li2017lower} using so-called `tilted walks' differs substantially from the construction with tilted interlacements in~\cite{li2014lowerBound}. 
\end{remark}

\appendix

\section{Comparison of Capacities}
In this appendix, we state and prove Proposition \ref{ComparisonCapacities}, which extends a certain scaling property of the Brownian capacity of a box to a set of distant boxes. The approach we are taking in order to prove our result is somewhat inspired by a similar comparison between the discrete capacity of a set of boxes and the Brownian capacity of its $\bbR^d$-filling, as performed in the appendix of \cite{nitzschner2017solidification} (in fact we will use this result at a certain stage in our proof).

We will first introduce some notation and recall some facts concerning capacities. Let $L \geq 1$ and $K \geq 100$ be integers and consider discrete boxes $B_z$ of size $L$ located at $z \in \bbZ^d$, i.e.
\begin{equation}
B_z = z + [0,L)^d \cap \bbZ^d.
\end{equation}
Their $\bbR^d$-fillings will be denoted by $\widehat{B}_z$, following the notation of the appendix in \cite{nitzschner2017solidification} i.e.
\begin{equation}
\widehat{B}_z = z + [0,L]^d (\subseteq \bbR^d).
\end{equation}
By convention, $B = B_0$ and $\widehat{B} = \widehat{B}_0$. Moreover, we define for $r \in (0,\tfrac{1}{4})$ the sets $B_z^{(r)}, B_{z,(r)} \subseteq \bbZ^d$ and $\widehat{B}_z^{(r)},\widehat{B}_{z,(r)} \subseteq \bbR^d$ as 
\begin{equation}\label{eq:BoxMod}
    \begin{aligned}
         &B_z^{(r)}  = z + [-rL, (1+r)L)^d \cap \bbZ^d,  \\ 
         &B_{z,(r)} = z + [rL, (1-r)L)^d \cap \bbZ^d, \\
        &\widehat{B}_z^{(r)} = z + [-rL, (1+r)L)^d, \\ 
        &\widehat{B}_{z,(r)} = z + [rL, (1-r)L)^d.
    \end{aligned}
\end{equation}
We will be interested in the capacity of a set of boxes located at points at mutual distance bigger than $KL$. More precisely, we let $\cC \subseteq \bbZ^d$ stand for a non-empty finite set of points with
\begin{equation}
\label{eq:PropertyDistances}
z \neq z' \in \cC \qquad \Rightarrow \qquad |z - z'|_\infty \geq KL,
\end{equation}
and we define the sets 
\begin{align}
&C  = \bigcup_{z \in \cC} B_z, & &  C^{(r)} = \bigcup_{z \in \cC} B_z^{(r)}, & &C_{(r)} = \bigcup_{z \in \cC} B_{z,(r)},  \\
&\Gamma = \bigcup_{z \in \cC}\widehat{B}_z, & & \Gamma^{(r)} = \bigcup_{z \in \cC}  \widehat{B}_z^{(r)}, & & \Gamma_{(r)} = \bigcup_{z \in \cC}  \widehat{B}_{z,(r)}.
\end{align}

Moreover, we will need the following perturbation and scaling results for equilibrium measures and capacities:

\begin{equation}
\label{eq:PerturbationformulaEqMeasures}
    \begin{minipage}{0.75\linewidth}
       For any $\delta > 0$ and $K \geq c(\delta)$ and $\cC$ fulfilling \eqref{eq:PropertyDistances}, it holds that $(1-\delta)\mu \leq e_C \leq (1+\delta)\mu$ with $\mu = \sum_{z \in \cC} e_C(B_z) \overline{e}_{B_z}$,
    \end{minipage}
\end{equation}
(see Proposition 1.5 of \cite{sznitman2015disconnection}), and 

\begin{equation}
\label{eq:AsymptoticEqualityCap}
\begin{split}
\liminf_{K, L \rightarrow \infty} \inf_{\cC} d \frac{\capa_{\bbZ^d}(C)}{\capa(\Gamma)} \geq 1, \qquad \limsup_{K, L \rightarrow \infty} \inf_{\cC} d \frac{\capa_{\bbZ^d}(C)}{\capa(\Gamma)} \leq 1.
\end{split}
\end{equation}
Finally, one knows that as $L \rightarrow \infty$ the capacities of $B$ and $\widehat{B}$ are equivalent, that is 
\begin{equation}
\label{eq:CapacitiesEquivalent}
a_L = d \frac{\capa_{\bbZ^d}(B)}{\capa(\widehat{B})} \rightarrow 1 \text{ as } L \rightarrow \infty,
\end{equation} 
see Lemma 2.2 of~\cite{bolthausen1993Critical} and p.301 of~\cite{spitzer2013principles}.

Note that \eqref{eq:PerturbationformulaEqMeasures}, \eqref{eq:AsymptoticEqualityCap} and \eqref{eq:CapacitiesEquivalent} continue to hold true uniformly in $r \in (0,\tfrac{1}{4})$ if we consistently replace all boxes $B_z$ and $\widehat{B}_z$ entering the quantities of interest by slight enlargements or diminutions according to \eqref{eq:BoxMod}.

The main result of this appendix is the following uniform comparison between the Brownian capacities of the sets $\Gamma$, $\Gamma^{(r)}$ and $\Gamma_{(r)}$:
\begin{prop}
\label{ComparisonCapacities}
\begin{align}
\label{eq:UpperComparison}
\sup_{\cC} \frac{\capa(\Gamma^{(r)})}{\capa(\Gamma)} \leq (1 + \delta_{K,L}) (1 + 2r)^{d-2}, \\
\label{eq:LowerComparison}
\inf_{\cC} \frac{\capa(\Gamma_{(r)})}{\capa(\Gamma)} \geq (1 - \delta_{K,L}) (1 - 2r)^{d-2},
\end{align}
where $0 \leq \delta_{K,L} \rightarrow 0$ as $K,L \rightarrow \infty$ and $\cC$ varies in the class of non-empty finite subsets of $\bbZ^d$ with the property \eqref{eq:PropertyDistances}. 
\end{prop}
\begin{proof}
We define the measure 
\begin{equation}
\nu(y) = \sum_{z \in \cC} e_{C^{(r)}}(B_z^{(r)}) \overline{e}_{B_z}(y),
\end{equation}
where $e_F$ is the equilibrium measure of a set $F \subseteq \bbZ^d$, see \eqref{eq:EqMeasure}, and $\overline{e}_F$ for a set with positive capacity stands for the normalized equilibrium measure of $F$, see~\eqref{eq:NormEqMeausre}. By construction, $\nu(\bbZ^d) = \capa_{\bbZ^d}(C^{(r)})$ and $\nu$ is supported on $C$. Moreover, we have for $x \in C$:
\begin{equation}
\begin{split}
&\sum_{y \in \bbZ^d} g(x,y) \nu(y)  = \sum_{z \in \cC} e_{C^{(r)}}(B_z^{(r)}) \sum_{y \in \bbZ^d} g(x,y) e_{B_z}(y) \frac{1}{\capa_{\bbZ^d}(B)} \\
&\ \stackrel{\eqref{eq:EquilibriumPotential}}{=} \frac{1}{\capa_{\bbZ^d}(B)}  \sum_{z \in \cC} e_{C^{(r)}}(B_z^{(r)}) \underbrace{h_{B_z}(x)}_{\leq h_{B_z^{(r)}}(x)} \\
&\ \stackrel{\eqref{eq:EquilibriumPotential}}{\leq}  \frac{\capa_{\bbZ^d}(B^{(r)})}{\capa_{\bbZ^d}(B)} \sum_{y \in \bbZ^d} g(x,y) \sum_{z \in \cC} e_{C^{(r)}}(B_z^{(r)}) \overline{e}_{B_z^{(r)}}(y) \\
&\ \stackrel{\eqref{eq:PerturbationformulaEqMeasures}}{\leq} \frac{\capa_{\bbZ^d}(B^{(r)})}{\capa_{\bbZ^d}(B)} \frac{1}{1-\delta} \sum_{y \in \bbZ^d} g(x,y) e_{C^{(r)}}(y) \leq  \frac{1}{1-\delta} \frac{\capa_{\bbZ^d}(B^{(r)})}{\capa_{\bbZ^d}(B)},
\end{split}
\end{equation}
for $K \geq c(\delta)$ large enough. Upon multiplication of this inequality with $e_C(x)$ and summation over $x \in \bbZ^d$, we obtain
\begin{equation}
\label{eq:ToDealWithMiddleTerm}
\begin{split}
\capa_{\bbZ^d}(C^{(r)}) & \leq \frac{1}{1-\delta} \frac{\capa_{\bbZ^d}(B^{(r)})}{\capa_{\bbZ^d}(B)} \capa_{\bbZ^d}(C) \\
& \stackrel{\eqref{eq:CapacitiesEquivalent}}{\leq}   \frac{1}{1-\delta} \widetilde{a}_L  \frac{\capa(\widehat{B}^{(r)})}{\capa(\widehat{B})} \capa_{\bbZ^d}(C)
\end{split}
\end{equation}
with $\widetilde{a}_L \rightarrow 1$ as $L \rightarrow \infty$.
Rearranging terms, we arrive at
\begin{equation}
\frac{\capa(\Gamma^{(r)})}{\capa(\Gamma)} \leq \frac{\capa(\Gamma^{(r)})}{\capa_{\bbZ^d}(C^{(r)})} \cdot  \frac{\capa_{\bbZ^d}(C^{(r)})}{\capa_{\bbZ^d}(C)} \cdot \frac{ \capa_{\bbZ^d}(C)}{\capa(\Gamma)}.
\end{equation}
Taking the supremum over $\cC$ fulfilling \eqref{eq:PropertyDistances} and using \eqref{eq:AsymptoticEqualityCap} for the first and last term in the above inequality and~\eqref{eq:ToDealWithMiddleTerm} one can find $0 \leq \delta_{K,L} \rightarrow 0$ as $K,L \rightarrow \infty$ such that
\begin{equation}
\sup_{\cC} \frac{\capa(\Gamma^{(r)})}{\capa(\Gamma)} \leq (1 + \delta_{K,L}) \frac{\capa(\widehat{B}^{(r)})}{\capa(\widehat{B})},
\end{equation}
for $K,L$ large enough. Since the Brownian capacity is translation invariant and fulfills $\capa(\alpha D) = \alpha^{d-2} \capa(D)$ for any $\alpha >0$ and $D \subseteq \bbR^d$ open or closed and bounded, the quotient of capacities in the last term is $(1+2r)^{d-2}$, which finishes the proof of \eqref{eq:UpperComparison}. For \eqref{eq:LowerComparison}, the proof is performed in a similar manner. 
\end{proof}
\textbf{Acknowledgements.}
The authors wish to thank Alain-Sol Sznitman for useful discussions and valuable comments at various stages of this project.


\begin{thebibliography}{10}
    \bibitem{bolthausen1993Critical}
    E.~Bolthausen and J.-D. Deuschel.
    \newblock Critical large deviations for {G}aussian fields in the phase
      transition regime, {I}.
    \newblock {\em The Annals of {P}robability}, 21(4):1876--1920, 1993.
    
    \bibitem{bolthausen2001entropic}
    E.~Bolthausen, J.-D. Deuschel, and G.~Giacomin.
    \newblock Entropic repulsion and the maximum of the two-dimensional harmonic
      crystal.
    \newblock {\em The Annals of {P}robability}, 29(4):1670--1692, 2001.
    
    \bibitem{bolthausen1995entropic}
    E.~Bolthausen, J.-D. Deuschel, and O.~Zeitouni.
    \newblock Entropic repulsion of the lattice free field.
    \newblock {\em Communications in {M}athematical {P}hysics}, 170(2):417--443,
      1995.
    
    \bibitem{chiarini2018entropic}
    A.~Chiarini and M.~Nitzschner.
    \newblock Entropic repulsion for the {G}aussian free field conditioned on
      disconnection by level-sets.
    \newblock {\em arXiv preprint arXiv:1808.09947}, 2018.
    
    \bibitem{deuschel1999pathwise}
    J.-D. Deuschel and G.~Giacomin.
    \newblock Entropic repulsion for the free field: Pathwise characterization in
      $d \geq 3$.
    \newblock {\em Communications in {M}athematical {P}hysics}, 206(2):447--462,
      1999.
    
    \bibitem{drewitz2014introduction}
    A.~Drewitz, B.~R{\'a}th, and A.~Sapozhnikov.
    \newblock {\em An {I}ntroduction to {R}andom {I}nterlacements}.
    \newblock Springer, 2014.
    
    \bibitem{drewitz2014local}
    A.~Drewitz, B.~R{\'a}th, and A.~Sapozhnikov.
    \newblock Local percolative properties of the vacant set of random
      interlacements with small intensity.
    \newblock {\em Ann. Inst. Henri Poincar{\'e} Probab. Stat.}, 50(4):1165--1197,
      2014.
    
    \bibitem{duminil2017sharp}
    H.~Duminil-Copin, A.~Raoufi, and V.~Tassion.
    \newblock Sharp phase transition for the random-cluster and {P}otts models via
      decision trees.
    \newblock {\em Ann. Math.}, 189:75--99, 2019.
    
    \bibitem{einmahl1989extensions}
    U.~Einmahl.
    \newblock Extensions of results of {K}oml{\'o}s, {M}ajor, and {T}usn{\'a}dy to
      the multivariate case.
    \newblock {\em Journal of {M}ultivariate {A}nalysis}, 28(1):20--68, 1989.
    
    \bibitem{fukushima2010dirichlet}
    M.~Fukushima, Y.~Oshima, and M.~Takeda.
    \newblock {\em Dirichlet forms and symmetric {M}arkov processes}, volume~19.
    \newblock Walter de Gruyter, 2010.
    
    \bibitem{gilbarg2015elliptic}
    D.~Gilbarg and N.S. Trudinger.
    \newblock {\em Elliptic partial differential equations of second order}.
    \newblock Springer, 2015.
    
    \bibitem{lawler2013intersections}
    G.F. Lawler.
    \newblock {\em Intersections of random walks}.
    \newblock Springer Science \& Business Media, 2013.
    
    \bibitem{li2017lower}
    X.~Li.
    \newblock A lower bound for disconnection by simple random walk.
    \newblock {\em The Annals of Probability}, 45(2):879--931, 2017.
    
    \bibitem{li2014lowerBound}
    X.~Li and A.-S. Sznitman.
    \newblock A lower bound for disconnection by random interlacements.
    \newblock {\em Electron. J. Probab.}, 19:1--26, 2014.
    
    \bibitem{li2015large}
    X.~Li and A.-S. Sznitman.
    \newblock Large deviations for occupation time profiles of random
      interlacements.
    \newblock {\em Probability Theory and Related Fields}, 161(1-2):309--350, 2015.
    
    \bibitem{nitzschner2018entropic}
    M.~Nitzschner.
    \newblock Disconnection by level sets of the discrete {G}aussian free field and
      entropic repulsion.
    \newblock {\em Electron. J. Probab.}, 23:1--21, 2018.
    
    \bibitem{nitzschner2017solidification}
    M.~Nitzschner and A.-S. Sznitman.
    \newblock Solidification of porous interfaces and disconnection.
    \newblock {\em To appear in J.\ Eur.\ Math.\ Society. Preprint available at
      arXiv:1706.07229}, 2017.
    
    \bibitem{port2012brownian}
    S.~Port and C.~Stone.
    \newblock {\em Brownian motion and classical potential theory}.
    \newblock Academic Press, 1978.
    
    \bibitem{spitzer2013principles}
    F.~Spitzer.
    \newblock {\em Principles of random walk}, volume~34.
    \newblock Springer Science \& Business Media, 2013.
    
    \bibitem{sznitman2009domination}
    A.-S. Sznitman.
    \newblock On the domination of random walk on a discrete cylinder by random
      interlacements.
    \newblock {\em Electron. J. Probab.}, 14:1670--1704, 2009.
    
    \bibitem{sznitman2010vacant}
    A.-S. Sznitman.
    \newblock Vacant set of random interlacements and percolation.
    \newblock {\em Ann. Math.}, 171:2039--2087, 2010.
    
    \bibitem{sznitman2012isomorphism}
    A.-S. Sznitman.
    \newblock An isomorphism theorem for random interlacements.
    \newblock {\em Electronic Communications in Probability}, 17:1--9, 2012.
    
    \bibitem{sznitman2012random}
    A.-S. Sznitman.
    \newblock Random interlacements and the {G}aussian free field.
    \newblock {\em The Annals of Probability}, 40(6):2400--2438, 2012.
    
    \bibitem{sznitman2015disconnection}
    A.-S. Sznitman.
    \newblock Disconnection and level-set percolation for the {G}aussian free
      field.
    \newblock {\em Journal of the Mathematical Society of Japan}, 67(4):1801--1843,
      2015.
    
    \bibitem{sznitman2017disconnection}
    A.-S. Sznitman.
    \newblock Disconnection, random walks, and random interlacements.
    \newblock {\em Probability Theory and Related Fields}, 167(1-2):1--44, 2017.
    
    \bibitem{sznitman2018macroscopic}
    A.-S. Sznitman.
    \newblock On macroscopic holes in some supercritical strongly dependent
      percolation models.
    \newblock {\em The Annals of Probability}, 47(4):2459--2493, 2019.
    
    \bibitem{teixeira2011fragmentation}
    A.~Teixeira and D.~Windisch.
    \newblock On the fragmentation of a torus by random walk.
    \newblock {\em Commun. Pure Appl. Math.}, 64(12):1599--1646, 2011.
    
    \bibitem{teschl2009mathematical}
    G.~Teschl.
    \newblock {\em Mathematical methods in quantum mechanics}, volume~99.
    \newblock American Mathematical Society, 2009.
    
    \bibitem{villani2008optimal}
    C.~Villani.
    \newblock {\em Optimal transport: old and new}, volume 338.
    \newblock Springer Science \& Business Media, 2008.
    
\end{thebibliography}
\end{document}